\documentclass{mcom-l}

\usepackage{amsmath,amssymb,ifthen}
\usepackage{fullpage,verbatim}
\usepackage{tikz,pgf}
\usepackage{graphicx,psfrag,subfigure}
\usetikzlibrary{arrows,automata}
\usetikzlibrary{positioning}

\tikzset{
	state/.style={
		rectangle,
		rounded corners,
		draw=black, line width = 1,
		minimum height=2em,
		inner sep=2pt,
		text centered,
	},
}
\tikzset{block/.style={rectangle split, draw, rectangle split parts=2, text centered, rounded corners, minimum height=4em, line width = 1}}
\newtheorem{theorem}{Theorem}[section]
\newtheorem{lemma}[theorem]{Lemma} 

\theoremstyle{definition}

\theoremstyle{remark}
\newtheorem{remark}[theorem]{Remark}

\numberwithin{equation}{section}

\def\E{\mathbb E}
\def\AA{\mathbb{ A}}

\def\P{{\mathbb P}}
\def\R{{\mathbb R}}
\def\N{{\mathbb N}}
\def\Z{{\mathbb Z}}

\def\id{{\boldsymbol{1}}}
\def\EE{{\mathcal E}}
\def\FF{{\mathcal F}}

\def\MM{{\mathcal M}}

\def\PP{{\mathcal P}}

\def\RR{{\mathcal R}}
\def\SS{{\mathcal S}}
\def\LL{{\mathcal L}}
\def\TT{{\mathcal T}}

\def\XX{{\mathcal X}}

\def\Dom{D}

\def\bx{\boldsymbol{x}}
\def\by{\boldsymbol{y}}
\def\bz{\boldsymbol{z}}

\def\norm#1#2{\|#1\|_{#2}}

\def\set#1#2{\big\{#1\,:\,#2\big\}}

\def\eps{\varepsilon}
\def\normL2#1#2{\|#1\|_{L^2(#2)}}

\newcommand{\dual}[3][]{#1\langle#2\,,\,#3#1\rangle}

\newcounter{constantsnumber}
\def\namec#1#2{%
 \ifthenelse{\equal{#1}{rel}}{C_{\rm rel}}{%
  \ifthenelse{\equal{#1}{mesh}}{C_{\rm mesh}}{%
  \ifthenelse{\equal{#1}{sz}}{C_{\rm sz}}{%
  \ifthenelse{\equal{#1}{dislocrel}}{C_{\rm dlr}}{%
  \ifthenelse{\equal{#1}{eff}}{C_{\rm eff}}{%
  \ifthenelse{\equal{#1}{main}}{C_{\rm V}}{%
  \ifthenelse{\equal{#1}{opt}}{C_{\rm opt}}{%
  \ifthenelse{\equal{#1}{normequiv}}{C_{\rm norm}}{%
  \ifthenelse{\equal{#1}{reliable}}{C_{\rm rel}}{%
  \ifthenelse{\equal{#1}{efficient}}{C_{\rm eff}}{%
  \ifthenelse{\equal{#1}{dlr}}{C_{\rm dlr}}{%
  \ifthenelse{\equal{#1}{stable}}{C_{\rm stab}}{%
  \ifthenelse{\equal{#1}{reduction}}{C_{\rm red}}{%
   \ifthenelse{\equal{#1}{unibound}}{C_{\rm hot}}{%
    \ifthenelse{\equal{#1}{hotConst}}{C_{\rm hot}}{%
   \ifthenelse{\equal{#1}{inverseK}}{C_{\rm K}}{%
  \ifthenelse{\equal{#1}{refined}}{C_{\rm ref}}{%
  \ifthenelse{\equal{#1}{estconv}}{C_{\rm est}}{%
  \ifthenelse{\equal{#1}{optimal}}{C_{\rm opt}}{%
  \ifthenelse{\equal{#1}{qo}}{C_{\rm qo}}{%
  \ifthenelse{\equal{#1}{mon}}{C_{\rm mon}}{%
  \ifthenelse{\equal{#1}{cea}}{C_{\mbox{\scriptsize C\'ea}}}{%
  \ifthenelse{\equal{#2}{newcounter}}{\refstepcounter{constantsnumber}\label{const#1}}{}C_{\ref{const#1}}}%
}}}}}}}}}}}}}}}}}}}}}}

\newcounter{contractionnumber}
\def\nameq#1#2{%
  \ifthenelse{\equal{#1}{reduction}}{q_{\rm red}}{%
  \ifthenelse{\equal{#1}{estconv}}{q_{\rm est}}{%
  \ifthenelse{\equal{#1}{cea}}{q_{\mbox{\scriptsize C\'ea}}}{%
  \ifthenelse{\equal{#2}{newcounter}}{\refstepcounter{contractionnumber}\label{contraction#1}}{}q_{\ref{contraction#1}}}%
}}}

\def\namer#1#2{%
  \ifthenelse{\equal{#1}{reduction}}{\rho_{\rm red}}{%
  \ifthenelse{\equal{#1}{estconv}}{\rho_{\rm est}}{%
  \ifthenelse{\equal{#1}{cea}}{\rho_{\mbox{\scriptsize C\'ea}}}{%
  \ifthenelse{\equal{#1}{qo}}{\rho_{\mbox{\scriptsize qo}}}{%
  \ifthenelse{\equal{#2}{newcounter}}{\refstepcounter{contractionnumber}\label{contraction#1}}{}\rho_{\ref{contraction#1}}}%
}}}}

\newtheorem{proposition}[theorem]{Proposition}
\newtheorem{corollary}[theorem]{Corollary}
\newtheorem{algorithm}[theorem]{Algorithm}
\newtheorem{assumption}[theorem]{Assumption}

\def\revision#1{#1}

\def\T{\mathbb T}

\begin{document}

\title{Recurrent Neural Networks\\ as Optimal Mesh Refinement Strategies}


%    author one information
\author{Michael Feischl}
\address{Institute for Analysis and Scientific Computing
TU Wien, Wiedner Hauptstraße 8-10, 1040 Vienna}
\curraddr{}
\email{michael.feischl@tuwien.ac.at}
\thanks{Both authors are supported by the Deutsche Forschungsgemeinschaft (DFG) 
through CRC 1173.}

%    author two information
\author{Jan Bohn}
\address{Institute for Applied and Numerical Mathematics, KIT, Englerstr. 2,
76131 Karlsruhe}
\curraddr{}
\email{jan.bohn@kit.edu}
\thanks{}

%    \subjclass is required.
\subjclass[2010]{Primary }

\date{}

\dedicatory{}

%    Abstract is required.
\begin{abstract}
We show that an optimal finite element mesh refinement algorithm for a 
prototypical elliptic PDE can be learned by a recurrent neural network
with a fixed number of trainable parameters independent of the desired accuracy 
and the input size, i.e., number of elements of the mesh. 
Moreover, for a general class of PDEs with solutions which are well-approximated by deep neural networks, we show that an optimal mesh refinement strategy can be learned by recurrent neural networks. This includes problems for which no optimal adaptive strategy is known yet.
\end{abstract}

\maketitle

\section{Introduction}\label{sec:intro}
Adaptive methods for finite element mesh refinement had tremendous impact on the 
scientific community both on the theoretical side as well as on the applied, 
engineering side.

Following the seminal works~\cite{bdd,stevenson07,ckns} on the adaptive finite 
element method, a multitude of papers extended the ideas to numerous model 
problems and applications, see e.g.,~\cite{ks,cn} for conforming 
methods,~\cite{rabus10,BeMao10,bms09,cpr13,mzs10}   
for nonconforming methods,~\cite{LCMHJX,CR2012,HuangXu} for mixed formulations, 
and~\cite{fkmp,gantumur,affkp,ffkmp:part1,ffkmp:part2} for boundary element 
methods 
(the list is not exhausted, see also~\cite{axioms} and 
the references therein). Quite recently,~\cite{stokesopt,nonsymm} also cracked 
non-symmetric and indefinite problems. All those works have in common that they 
use a standard adaptive refinement algorithm
of the form
\begin{align*}
\fbox{Solve}\longrightarrow\fbox{Estimate}\longrightarrow\fbox{Mark}
\longrightarrow\fbox{Refine}
\end{align*}
where an error estimator is computed from the current solution and then used to 
refine certain elements of the mesh. The actual refinement of the individual 
elements of the mesh is usually done with an algorithm called newest-vertex 
bisection (see, e.g.,~\cite{stevenson08}). A general drawback of adaptive mesh 
refinement methods is often their very specific area of application and their implementational overhead involved in the error 
estimation and choosing elements which to refine.

This  encourages the development of black-box tools which can be adapted to a wide range 
of problems. In view of the huge practical success of recurrent neural networks 
(RNNs) in various applications and their flexibility in terms of the length of the input sequence (after all we do not want to retrain the network meshes of different sizes), they might provide exactly the required 
black-box tool. The most prominent examples of RNNs are Long-Term-Short-Term 
memory approaches proposed in~\cite{lstm} and since then hugely successful in 
practical applications, e.g., for time-series interpretation~\cite{timeseries}, 
speech recognition~\cite{speech}, speech synthesis~\cite{speech2}, and even 
surgical robot control~\cite{robot}. Very roughly, a recurrent neural network 
has the following structure
\begin{center}
	\begin{tikzpicture}[->,>=stealth']
	\node[state] (X1) 
	{%
		$X_1$
	};
	\node[state,above of =X1,yshift=1cm] (DNN1) 
	{DNN};
	\node[state,above of =DNN1,yshift=1cm] (Y1) 
	{$Y_1$};
	\node[state,right of = X1, xshift=1cm] (X2) 
	{%
		$X_2$
	};
	\node[state,above of =X2,yshift=1cm] (DNN2) 
	{DNN};
	\node[state,above of =DNN2,yshift=1cm] (Y2) 
	{$Y_2$};
	\node[state,right of = X2, xshift=1cm] (X3) 
	{%
		$X_3$
	};
	\node[state,above of =X3,yshift=1cm] (DNN3) 
	{DNN};
	\node[state,above of =DNN3,yshift=1cm] (Y3) 
	{$Y_3$};
	\node[state,right of=DNN3,xshift=1cm, draw = white ] (DNN4) 
	{\hspace{5mm}$\cdots$\hspace{5mm}\,};
	\node[state,right of=X3,xshift=1cm, draw = white ] (A) 
	{\hspace{5mm}$\cdots$\hspace{5mm}\,};
	\node[state,right of=Y3,xshift=1cm, draw = white ] (B) 
	{\hspace{5mm}$\cdots$\hspace{5mm}\,};
	\node[state,right of =DNN4,xshift=1cm] (DNNn) 
	{DNN}; 
	\node[state, below of = DNNn, yshift =-1cm] (Xn) 
	{%
		$X_n$
	};	
	\node[state,above of =DNNn,yshift=1cm] (Yn) 
	{$Y_n$};
	\path (X1) edge  node[anchor=south,above]{} (DNN1)
	(DNN1) edge node[anchor=south,above]{} (Y1)
	 (X2) edge  node[anchor=south,above]{} (DNN2)
	(DNN2) edge node[anchor=south,above]{} (Y2)
	 (X3) edge  node[anchor=south,above]{} (DNN3)
	 (DNN3) edge node[anchor=south,above]{} (Y3)
	(DNN1) edge node[anchor=south,above]{} (DNN2)
	(DNN2) edge node[anchor=south,above]{} (DNN3)
	(DNN3) edge node[anchor=south,above]{} (DNN4)
	(DNN4) edge node[anchor=south,above]{} (DNNn)
	(Xn) edge  node[anchor=south,above]{} (DNNn)
	(DNNn) edge node[anchor=south,above]{} (Yn);
	\end{tikzpicture}
\end{center}
where the $X_1,X_2,\ldots,X_n$ denote a (vector valued) input sequence and the 
$Y_1,Y_2,\ldots,Y_n$ a (vector valued) output sequence. The block DNN denotes a 
standard deep neural network which maps the input state to the output state, but 
may also use hidden intermediate states from the previous iteration of the 
network. The major advantage of this structure compared to a fully connected DNN 
over all $n$ input states is that the weights of the DNNs are shared for all 
iterations. This means that an arbitrary long input sequence can be treated with 
a DNN depending only on a bounded number of trainable parameters. We will use 
this fact in order to construct a network whose parameter count does not depend 
on the number of elements of the current adaptive mesh.

The idea and question motivating this work is the following: Can we replace the 
steps $\fbox{Estimate}\longrightarrow\fbox{Mark}$ by a recurrent neural network 
$\fbox{ADAPTIVE}$ in order to achieve similar (or better) results than state of 
the art adaptive mesh refinement algorithms?

We answer this question in two ways: The first main result in Section~\ref{sec:main1} shows that an RNN ADAPTIVE
can be trained to achieve at least the performance of adaptive algorithms which are known to be optimal for
second order elliptic PDEs. The second main result in Section~\ref{sec:main2} shows for a broad class of problems that as long as
the exact solution of a PDE can theoretically be efficiently approximated by a RNN, the RNN ADAPTIVE
can be trained to produce optimally refined meshes. Roughly speaking, the present work shows that black-box mesh refinement by use of RNNs is at least as good as current optimal mesh refinement technology and can even achieve optimal results in areas which are not yet covered by the theory of adaptive mesh refinement. 

% First, while asymptotic optimal error estimators are known for some problems, 
% their preasymptotic performance is largely unknown. It is the hope of the authors 
% that automated machine learning could come up with an estimate--refine scheme 
% which drastically improves the preasymptotic behavior of known error estimators.
% 
% Second, for many problems there are no optimal adaptive algorithms known, or 
% sometimes there are not even error estimators available. By training a neural 
% net built from
% blue prints laid out in this work, one can obtain optimal mesh 
% refinement algorithms for those hard problems. Examples for those problems are, 
% e.g., time-dependent equations like the Landau-Lifshitz-Gilbert equations or the 
% Navier-Stokes equations. We show that this is actually possible for a quite general class of problems.

% \medskip
% 
% The present work shows that a mildly complex RNN can indeed emulate an optimal 
% adaptive mesh refinement strategy. The major difficulty is that the number of 
% trainable parameters of the RNN should not depend on the input size, i.e., the 
% number of mesh elements of the current iteration of the adaptive algorithm. This 
% ensures that, once trained, a RNN emulating the adaptive algorithm can be used 
% for a variety of problem sizes and input parameters. 

\medskip

The remainder of the work is structured as follows: Section~\ref{sec:main} 
introduces the model problem, provides definitions of RNNs and optimal adaptive 
algorithms, and states the main results.
Section~\ref{sec:conclude} discusses the applicability of the main results as well as the implementation of the training process.
Section~\ref{sec:construction} provides all the sub assemblies for the RNN which 
emulates the adaptive algorithm. Sections~\ref{sec:complete} and~\ref{sec:if} contain the proofs of the main results. A final 
Section~\ref{sec:numerics} underlines the theoretical findings by some numerical 
experiments.

\section{Model Problem \& Main Results}\label{sec:main}
On the open Lipschitz domain $\Dom\subset \R^d$, $d=2,3$, we consider a 
prototypical PDE of the form
\begin{align}\label{eq:weak}
\begin{split}
\LL u &= f\quad\text{in 
}\Dom,\\
u&=0\quad\text{on }\partial\Dom,
\end{split}
\end{align}
where $\LL\colon \XX \to \XX^\star$ is an isomorphism for some Hilbert space $\XX$.
With discrete spaces $\XX(\TT)\subset\XX$ based on some triangulation $\TT$ of $\Dom$, this allows us to write down the discrete 
form of the equation: Find $U_\TT\in\XX(\TT)$ such that
\begin{align}\label{eq:discrete}
\dual{\LL U_\TT}{V}=f(V)\quad\text{for all }V\in \XX(\TT).
\end{align}
We assume that also $\LL|_{\XX(\TT)}$ is an isomorphism to obtain a unique discrete solution.
\subsection{Optimal mesh refinement}
We consider an initial regular and shape regular triangulation $\TT_0$ of $\Dom$ into compact simplices $T\in\TT_0$.
Such that $\TT$ partitions $\Dom$ into compact simplices
such that the intersection of two elements $T\neq T'\in\TT$ is either: a common 
face, a common node, or empty.
In the recent literature~\cite{axioms,stevenson07,ckns}, mesh refinement algorithms are steered by an error estimator $\rho(\TT) = \rho(\TT,U_\TT,f) = \sqrt{\sum_{T\in\TT} \rho_T^2}$ which satisfies $\rho(\TT,U_\TT,f)\approx \norm{u-U_\TT}{\XX}$ and have the following basic structure:
\begin{algorithm}\label{alg:adaptive}
	\textbf{Input: } Initial mesh $\TT_0$, parameter $0<\theta<1$, tolerance $\eps>0$.\\
	For $\ell=0,1,2,\ldots$ do:
	\begin{enumerate}
		\item Compute $U_{\ell}:=U_{\TT_\ell}$ from~\eqref{eq:discrete}.
		\item Compute error estimate $\rho_T$ for all $T\in\TT_\ell$. If $\sum_{T\in \TT_\ell} \rho_T^2\leq \eps^2$, stop.
		\item Find a set $\MM_\ell\subseteq \TT_\ell$ of minimal 
cardinality such that
		\begin{align}\label{eq:doerfler}
		\sum_{T\in\MM_\ell}\rho_T^2\geq \theta 
\sum_{T\in\TT_\ell}\rho_T^2.
		\end{align}
		\item Use newest-vertex-bisection to refine at least the 
elements in $\MM_\ell$ and to obtain a new mesh $\TT_{\ell+1}$.
	\end{enumerate}
	\textbf{Output: } Sequence of adaptively refined meshes $\TT_\ell$ and 
corresponding approximations $U_\ell\in{\XX}(\TT_{\ell})$ such that $\rho(\TT_\eps)\leq \eps$ for final step $\TT_\eps := \TT_L$ and $L\in\N$.
\end{algorithm}

We consider the following notion of optimality of the mesh refinement algorithm: 
Let $\T$ denote the set of all possible meshes which can be generated by 
iterated application of newest-vertex-bisection to the initial mesh $\TT_0$. 
Then, the maximal possible convergence rate $s>0$ is defined by the maximal 
$s>0$ such that
\begin{subequations}\label{eq:opt}
\begin{align}
\sup_{N\in\N} \inf_{\TT\in\T\atop \#\TT-\#\TT_0\leq N} \rho(\TT,U_\TT,f)N^s 
<\infty.
\end{align}
We call Algorithm~\ref{alg:adaptive} optimal if it satisfies
\begin{align}
\sup_{0<\eps\leq 1}\#\TT_\eps \rho(\TT_\eps)^{1/s} <\infty
\end{align}
\end{subequations}
 for the same rate $s$.
 
 \bigskip

 The main goal of this work is to prove that a particular type of neural network can be trained to perform
the steps (2) and (3) of Algorithm~\ref{alg:adaptive} in an optimal way without any further knowledge about $\LL$. We show that
this is possible for second order elliptic operators $\LL$ in Section~\ref{sec:main1} and for a much broader class of problems
in Section~\ref{sec:main2}.
 
 %The main theorem of interest for the present work is the 
% following. The following Theorem is the main result in many works on optimal adapive mesh refinement and it holds for numerous model problems (see, e.g., ~\cite{nonsymm,axioms,ckns, stevenson07} for linear second order elliptic operators or~\cite{fembemopt,stokesopt} for transmission problems and the Stokes problem).
% \begin{theorem}
% {In the situation of Section \ref{sec:mainresult1} with constant coefficients $A,b,c$}, Algorithm~\ref{alg:adaptive} applied to~\eqref{eq:weak}--\eqref{eq:discrete} with error estimator~\eqref{eq:errest}
% is optimal in the sense of~\eqref{eq:opt} provided that $\theta$ is sufficiently 
% small.
% \end{theorem}

\subsection{Definition of Deep Neural Networks}
We consider standard ReLU networks $B$ which can be defined as follows: For a 
given input $x\in \R^{s_0}$ and weight matrices $W_j\in\R^{s_{j+1}\times s_j}$,
$j=0,\ldots,d$,
we define the output $y\in \R^{s_{d+1}}$ as
\begin{align*}
y:=B(x):=W_{d} \phi(W_{d-1}\phi(W_{d-2}(\cdots \phi(W_0 x)\cdots ))),
\end{align*}
where the activation function is defined as $\phi(y):=\max(y,0)$ and is applied 
entry wise to vector valued inputs. 
A DNN is said to have depth $d$ and width $\max_{j=0,\ldots,d+1} s_j$. The number of weights is given by $\sum_{j=0}^{d}s_{j+1}s_j$. We do not 
specify biases explicitely as we can always assume an additional constant input
state $x_0$. Clearly, compositions of $+$, $-$, $\min$, $\max$, and $|\cdot|$ can be constructed as DNNs. Moreover, given two DNNs $B_1, B_2$, their composition is also a DNN.
This is implicitly used in the following. We define the complexity of the DNN by the number of weights.

\subsection{Definition of Basic Recurrent Neural Networks}
 A RNN $B$ is a deep neural network $B$ with output size $s'\in\N$ 
and input size $s+s'$, $s\in\N$. For reasons that become clear below, we denote this as a basic RNN.
 The DNN $B$  is applied to each entry of a (vector-valued) sequence 
$\bx=(x_1,x_2,\ldots,x_n)\in\R^{s\times n}$
 and returns another (vector-valued) sequence 
$\by=(y_1,y_2,\ldots,y_n)\in\R^{s'\times n}$, $s,s'\in\N$. Additionally, the 
previous output state $y_{i-1}$ is
 fed into $B$ as an inpute state, i.e.,
 \begin{align*}
y_i:=B(x_i,y_{i-1}):=W_{d} \phi(W_{d-1}\phi(W_{d-2}(\cdots \phi(W_0 
\binom{x_i}{y_{i-1}})\cdots ))), \quad i=1,\ldots,n.
\end{align*}
The weight matrices $W_j$ and hence the complexity of $B$ is independent of 
$n\in\N$. The number of weights of a basic RNN is just the number of entries in the weight matrices of the underlying DNN, width and depth are defined analogously. Hence, the complexity of a RNN is defined as the complexity of the underlying DNN. We also use the expression size synonymous to complexity. 
For $i=1$, we allways assume that $y_0=0\in \R^{s'}$.
For example, a simple summation over the seqeuence $\bx\in\R^{{ 1\times n}}$ can be 
realized by 
\begin{align*}
 y_i=B(x_i,y_{i-1}):= x_i+y_{i-1}=\begin{pmatrix} 1 & 
-1\end{pmatrix}\phi\Big(\begin{pmatrix} 1 & 1\\-1 & 
-1\end{pmatrix}\binom{x_i}{y_{i-1}}\Big).
\end{align*}
The last entry of $\by\in\R^{{1\times  n}}$ contains the sum 
$y_n=\sum_{i=1}^nx_i$.

\subsection{Fixed number of independent weights} {As done in the previous subsection,}
by applying the same DNN $B$ to different parts of an input vector 
$x\in\R^{n_0}$, we may construct DNNs with arbitrary width but a fixed number of 
independent weights, i.e.,
the weights of $B$. By stacking the networks on top of each other, i.e., $B\circ 
B\circ \ldots \circ B(x)$ in case of $s_{d+1}=s_0$, we may create DNNs with 
arbitrary depth but still a fixed number of independent weights. \\
The distinction between number of independent weights and number of
total weights is made since in the constructions below, {the size of some networks grows} logarithmically in the accuracy, however,
they are just iterations of the same basic building block and hence the number of independent weights stays constant.
This might benefit the training process, as the search space remains of constant size. On the other hand, the topology of the search space changes as the number of total weights grows. Therefore, further research is required on whether  bounded number of independent weights can be used to speed up the training.
\subsection{Definition of Deep Recurrent Neural Networks}\label{sec:rnn}
We adopt a more general definition of RNNs in this work. We allow ourselves to 
deal with finite concatenations of those basic building blocks
from the previous section,
 i.e., in our notion a RNN is a finite stack of $m$ basic RNNs $B_i$ in the 
sense
\begin{align*}
B_m\circ B_{m-1} \circ \ldots \circ B_2 \circ B_1(\bx),
\end{align*}
i.e., the output sequence of $B_1$ is fed into $B_2$ and so on.
Additionally, we allow that the input $\bx$ is initialized by the last entry of 
the output sequence of the previous network. 
So, in its most general form, the combination $\by'=B_2\circ B_1(\bx)$ of two basic RNNs 
$B_1,B_2$ can be written as
\begin{align*}
\bx\in\R^{s_1\times n}\mapsto \by=B_1(\bx)\in\R^{s_1'\times n}\quad \text{and} 
\quad \bx'=\begin{pmatrix}
y_1 & y_2 & \cdots & y_n\\
y_n & 0 &\cdots & 0 
\end{pmatrix}\in \R^{2s_1'\times n}\mapsto \by':=B_2(\bx')\in\R^{s_2'\times n}.
\end{align*}
We may write vector valued sequences $\bx\in\R^{s\times n}$ as vectors of sequences, i.e., $(\bx_1,\ldots,\bx_s)$.
This choice of neural network class might seem arbitrary, however, it gives us much more freedom when 
constructing the networks and does not sacrifice 
the simplicity of the function class. This means that the complexity {(defined as the sum over the complexities of the underlying basic RNN's)} of a 
stacked RNN is still independent of the sequence length $n$.
Similar constructions of deep RNNs (stacked RNNs) are considered in~\cite{dRNN1,dRNN2,dRNN3}.

\subsection{Elementary operations with deep RNNs}
We illustrate some constructions which will be used implicitly in the proofs below:
\begin{itemize}
\item Identity-DNN: The identity function ${\rm id}(x)=\max(x,0)-\max(-x,0)$ can be emulated by a DNN of depth $d\geq 1$ with the following weight matrices ($1$ stands for the identity matrix of the correct size): $W_0:=\begin{pmatrix}1\\ -1\end{pmatrix}$, $W_i:=\begin{pmatrix}1&0\\0&1\end{pmatrix}$ for $i=1,\ldots,d-1$, and $W_d:=\begin{pmatrix}1&-1\end{pmatrix}$. Similarly, we can define Identity-RNN's.

\item Matrix multiplication: The multiplication with a matrix $y=Mx$ for $M\in\R^{s'\times s}$ can be constructed as $W_0:=\begin{pmatrix}1\\ -1\end{pmatrix}$, $W_1:=\begin{pmatrix}M&-M\end{pmatrix}$, and $W_2:=\begin{pmatrix}1 &-1\end{pmatrix}$.

\item Applying DNN/RNN's simultaneously: If we want to compute the output of two DNN's $z_1=A(x_1)$, $z_2=B(x_2)$ at once, we can define the DNN $(z_1,z_2)=C(x_1,x_2):=(A(x_1),B(x_2))$ by
\begin{align*}
 W_{C,i}:=\begin{pmatrix} W_{A,i} & 0 \\ 0 & W_{B,i}\end{pmatrix},
\end{align*}
where $0$ denotes the zero matrix of appropriate size.
In case the DNN's $A$ and $B$ have different depths, we use identity DNN's to extend $A$ and $B$ to equal depth (note that the resulting depth satisfies $d_c\leq \max(d_A,d_B)+2$ since an identity DNN has at least two layers). Similarly we can apply RNN's simultaneously as long as the input lengths coincide. If we apply  $m\in\N$ networks $B_1,\ldots, B_m$ simultaneously, the resulting network $C$ satisfies
$d_C\leq \max_{i=1,\ldots,m} d_{B_i} +3$ and the width of $C$ is bounded by the sum of the widths of the $B_i$.

\item Any given deep RNN $\by=B(\bx)$ can be extended such that it copies an additional input variable to the output, i.e., there exists $\widehat B$ with comparable complexity to $B$ such that $(\by,\bx_2) = \widehat B(\bx_1,\bx_2)$. In case $B$ is a basic RNN, this can be achieved by, e.g., defining $\widehat B(x_{1,i},y_{i-1},x_{2,i}):=(B(x_{1,i},y_{i-1}),{\rm id}(x_{2,i}))$. If $B$ is a deep RNN, the same construction can be applied to all the basic RNN's that compose $B$.

\item A basic RNN $B$ which turns a given sequence $\bx=(x,0,\ldots,0)\in\R^n$ into the constant sequence $\by=(x,x,\ldots,x)\in\R^n$ can be defined by $y_i=B(x_i,y_{i-1})=x_i+y_{i-1}$. Similarly one can add storage sequence for specific values inside an RNN.

\item Composition of RNN's and DNN's: 
\begin{itemize}
\item
 (basic RNN)$\circ$(DNN): The composition of a basic RNN with a DNN is again a RNN, by directly composing the underlying DNN of the RNN and the DNN.

\item (DNN)$\circ$(basic-RNN): For a DNN $y=B_2(x)$ and a basic RNN $y_i = B_1(x_i,y_{i-1})$, we may construct the composition by $(z_i,y_i) = (B_2\circ B_1(x_i,y_{i-1}), B_1(x_i,y_{i-1}))$ (note that the second entry in the output sequence is necessary for the correct evaluation of $B_1$).
\end{itemize}

\item RNNs as DNNs:
A RNN $B$ can be interpreted as a DNN $B'$. This 
means that we fix the input size $n$ of $B$ and consider the resulting neural 
network $B'$ which has $n$-times the width and depth of $B$ with a total number of weights of $n^3$ times the number of weights of $B$, as can be seen from: 
\begin{center}
	\begin{tikzpicture}[->,>=stealth',scale=0.7,every node/.style={scale=0.7}]
	\node[state] (X1) 
	{%
		$X_1$
	};
	\node[state,above of =X1,yshift=1cm] (DNN1) 
	{DNN};
	\node[state,above of =DNN1,yshift=1cm] (Y1) 
	{$Y_1$};
	\node[state,right of = X1, xshift=1cm] (X2) 
	{%
		$X_2$
	};
	
	\node[state,right of =Y1,xshift=1cm] (DNN2) 
	{DNN};
	\node[state,above of  =DNN2,yshift=1cm] (Y2) 
	{$Y_2$};
	\node[state,right of = X2, xshift=1cm] (X3) 
	{%
		$X_3$
	};
	\node[state,right of = Y2,xshift=1cm] (DNN3) 
	{DNN};
	\node[state,above of =DNN3,yshift=1cm] (Y3) 
	{$Y_3$};
	\node[state,right of= Y3,xshift=1cm, draw = white ] (DNN4) 
	{\hspace{5mm}$\cdots$\hspace{5mm}\,};
	\node[state,right of=X3,xshift=1cm, draw = white ] (A) 
	{\hspace{5mm}$\cdots$\hspace{5mm}\,};
	\node[state,right of=Y3,xshift=1cm, draw = white ] (B) 
	{\hspace{5mm}$\cdots$\hspace{5mm}\,};
	\node[state,right of =DNN4,xshift=1cm] (DNNn) 
	{DNN}; 
	\node[state, right of = A, xshift =1cm] (Xn) 
	{%
		$X_n$
	};	
	\node[state,above of =DNNn,yshift=1cm] (Yn) 
	{$Y_n$};
	\path (X1) edge  node[anchor=south,above]{} (DNN1)
	(DNN1) edge node[anchor=south,above]{} (Y1)
	 (X2) edge  node[anchor=south,above]{} (DNN2)
	(DNN2) edge node[anchor=south,above]{} (Y2)
	 (X3) edge  node[anchor=south,above]{} (DNN3)
	 (DNN3) edge node[anchor=south,above]{} (Y3)
	(DNN1) edge node[anchor=south,above]{} (DNN2)
	(DNN2) edge node[anchor=south,above]{} (DNN3)
	(DNN3) edge node[anchor=south,above]{} (DNN4)
	(DNN4) edge node[anchor=south,above]{} (DNNn)
	(Xn) edge  node[anchor=south,above]{} (DNNn)
	(DNNn) edge node[anchor=south,above]{} (Yn);
	\end{tikzpicture}
\end{center}
However, the number of independent weights is determined only by the number of weights in $B$ and hence independent of $n$.
\item RNN's with input $\bx=(x,0,\ldots,0)\in\R^n$ and output $\by=(y_1,\ldots,y_n)\in\R^n$ that are interpreted as DNN's can be written as DNN's with one dimensional input $x\in\R$ and output $y_n\in\R$, by multiplication with the matrices $(1,0,\ldots,0)$ and $(0,\ldots,0,1)^T$.
\end{itemize}

% To simplify notation, we will often write $B(\bx)=B(\bx^1,\bx^2,\ldots,\bx^s)$ 
% when dealing with vector valued input sequences $\bx\in\R^{s\times n}$.
% Moreover, we will not explicitely write down the weight matrices $W_j$ whenever 
% their construction is clear from the formulas (i.e., concatenations of addition, 
% substraction, 
% and $\phi$).

\subsection{Main Result 1} \label{sec:main1}
On the open Lipschitz domain $\Dom\subset \R^d$, $d=2,3$, we consider a 
prototypical operator $\LL$ of the form
\begin{align}\label{eq:weak0}
\begin{split}
\LL u = -{\rm div}(A\nabla u) + b\cdot \nabla u + cu
\end{split}
\end{align}
where $\LL$ has coefficients $A,b,c\in L^\infty(\Dom)$ such that the 
associated bilinear form
\begin{align*}
a(u,v):=\dual{\LL u}{v}\quad\text{for all }u,v\in H^1_0(\Dom)
\end{align*}
satisfies $a(u,v)\leq C \norm{u}{H^1(\Dom)}\norm{v}{H^1(\Dom)}$ as well as 
$a(u,u)\geq C^{-1}\norm{u}{H^1(\Dom)}^2$ for some constant $C>0$. The 
Lax-Milgram lemma guarantees a unique solution $u\in H^1(\Dom)$ 
of~\eqref{eq:weak} and~\eqref{eq:discrete}.
 On a triangulation $\TT$, we define the Ansatz and 
test spaces
\begin{align*}
 \PP^p(\TT)&:=\set{v\in L^2(\Dom)}{v|_T\text{ is a polynomial of degree }\leq 
p,\,T\in\TT}\\
\XX(\TT):=\SS^p(\TT)&:=\PP^p(\TT)\cap H^1_0(\TT)
\end{align*}
for a polynomial degree $p\in\N_0$. We set $r(p,d):={\rm dim}(\PP^p)$, the dimension of the space of polynomials of degree $p$ in $d$ dimensions.
The residual based error estimator for the given problem reads
\begin{subequations}\label{eq:errest}
\begin{align}
\rho_T^2:=\rho_T(\TT,U_\TT,f)^2:={\rm diam}(T)^2\norm{f-\LL U_\TT}{L^2(T)}^2 + 
{\rm diam}(T)\norm{[n\cdot A\nabla U_\TT]}{L^2(\partial T\cap \Dom)}^2
\end{align}
on each element $T\in\TT$ with normal vector $n$ on the boundary $\partial T$ 
and $[\cdot]$ denoting the jump over element faces, and the overall estimator is
the sum of the elementwise contributions, i.e.,
\begin{align}
\rho(\TT):=\rho(\TT,U_\TT,f):=\sqrt{\sum_{T\in\TT}\rho_T(\TT,U_\TT,f)^2}.
\end{align}
\end{subequations}
To avoid having to deal with data oscillations, we restrict ourselves to the simple case of $A|_T,b|_T,c|_T, f|_T\in \PP^p(T)$ for all $T\in\TT_0$.
Obviously, the error estimator $\rho_T$ depends on the values of $U_\TT$ on the 
whole patch $\omega_T:=\set{T'\in\TT}{T'\text{ shares a face with } T}$.
% The error estimator is reliable and efficient in the sense
% \begin{align*}
% C_{\rm rel}^{-1}\norm{u-U_\TT}{H^1(\Dom)}\leq \rho(\TT,U_\TT,f)\leq C_{\rm 
% eff}\norm{u-U_\TT}{H^1(\Dom)},
% \end{align*}
% for constants $C_{\rm rel},C_{\rm eff}>0$ which depend only on the shape 
% regularity of $\TT$ and on $p$.
 The main goal of the first part of this work is to show that RNNs of almost constant size are 
 capable of performing optimal mesh refinement for the PDE given 
 in~\eqref{eq:weak}.
To that end, we construct a RNN which performs steps~(2)--(3) of 
Algorithm~\ref{alg:adaptive}.

\begin{assumption}\label{roundoff}
We assume all numbers $x,y\in\R$ occuring in computations of the following algorithms satisfy the following: If $x\neq y$, there holds $|x-y|\geq 2^{-n_{\rm min}} \max\{|x|,|y|\}$ for some universal exponent $n_{\rm min}\in\N$.
This assumption allows us to emulate step functions with neural networks which are continuous by construction. The assumption is satisfied in floating point number systems such as \texttt{double}-arithmetic, where $n_{\rm min}$ corresponds to the accuracy in terms of the number of digits, for \texttt{double}-arithmetic, it is $n_{\rm min}=52$.
 \end{assumption}

\begin{theorem}\label{thm:refinement}
	For given $\eps>0$, there exists a deep RNN {\rm 
ADAPTIVE} 
 which takes a vector-valued input sequence 
$\bx\in\R^{({2(d+1)d +(d+3)r(p,d)})\times  \#\TT}$ such that $x_i$ contains the nodes 
of the elements $T'\in\omega_{T_i}$ for $T_i\in\TT$ and the corresponding 
polynomial expansions of $U_{T'}$ and {$f|_{T_i}$}.
%Assuming a shuffled input 
%	sequence $\bx'=(x_{\pi(1)},\ldots,x_{\pi(\#\TT)})$ for some randomly 
%chosen permutation $\pi$,
	The output $\by:={\rm ADAPTIVE}(\bx)\in \R^{\#\TT}$ satisfies
	\begin{align}\label{eq:approxdoerfler}
	\sum_{T_i\in \TT\atop y_{i}>0} \widetilde\rho_{T_i}^2\geq \theta 
\sum_{T\in\TT}\widetilde\rho_T^2
	\end{align}
	%with probability $1-C\exp(-\alpha)$ 
for estimators
	$\widetilde \rho_T$ which satisfy
	\begin{align*}
	 |\rho_T^2-\widetilde\rho_T^2|\leq C_{\rm ada}\frac{\eps}{\#\TT}\quad\text{for all 
}T\in\TT.
	\end{align*}
	with a uniform constant $C_{\rm ada}>0$. 
Moreover, the number of positive entries in $\by$ is minimal
	in order to satisfy~\eqref{eq:approxdoerfler}. 
	The RNN has a fixed number of independent weights. 
	The RNN can be constructed with a total number of weights of  $O( (n_{\rm min}+\log(\#\TT)+|\log(\eps)|+|\log(\norm{\bx}{\infty})|)^4)$  (see Figure~\ref{fig:RNNADAPTIVE} for the precise structure).  The 
magnitude of the weights is  $\mathcal{O}(1)$. Additionally, for a given overall tolerance $\eps_{{\rm tol}}^2>0$, it holds $\by\le 0,$ as soon as the tolerance $\tilde\rho(\TT,U,f)^2:=\sum_{T_i\in \TT } \widetilde\rho_{T_i}^2\le \eps_{{\rm tol}}^2$ is reached.  
\end{theorem}
We refer to Section~\ref{sec:complete} for the proof of the Theorem. This result 
suggests the following algorithm:
\begin{algorithm}\label{alg:RNN}
	\textbf{Input: } Initial mesh $\TT_0$, tolerance $\eps_{\rm tol}>0$. \\
	For $\ell=0,1,2,\ldots$ do:
	\begin{enumerate}
		\item Compute $U_{\ell}$ from~\eqref{eq:discrete}.
		\item Apply $\by={\rm ADAPTIVE}(\bx)$ as defined in 
Theorem~\ref{thm:refinement}.
		\item Use newest-vertex-bisection with mesh closure to refine the elements 
$T_i\in\TT_\ell$ with $y_i> 0$ to obtain a new mesh $\TT_{\ell+1}$ or stop if $\by\leq0$.
		 %\item  Stop if $\sum_{T_i\in \TT} \widetilde\rho_{T_i}^2 \le 2\eps^2_{\rm tol}/3$.%$\by\leq 0$.
	\end{enumerate}
	\textbf{Output: } Sequence of adaptively refined meshes $\TT_\ell$ and 
corresponding approximations $U_\ell\in\SS^p(\TT_{\ell})$ such that ${\tilde\rho}(\TT_{\eps_{\rm tol}})\leq\eps_{\rm tol}$
for final step $\TT_{\eps_{\rm tol}}:= \TT_L$.
\end{algorithm}

From the previous theorem, we derive the following consequence.
\begin{corollary}\label{cor:refinement}
Given $\eps>0$ in Theorem~\ref{thm:mark}, Algorithm~\ref{alg:RNN} is optimal in the sense
 \begin{align*}
 \sup_{ \sqrt{4C_{\rm ada}\eps/\theta}\leq \eps_{\rm tol}\leq 1} \#\TT_{\eps_{\rm tol}} {\eps_{\rm tol}}^{1/s} \leq C<\infty
 \end{align*}
 with the maximal rate $s>0$ from~\eqref{eq:opt} and $C>0$ independent of $\eps$ and $\eps_{\rm tol}$.
\end{corollary}
\begin{proof}
We may assume $C_{\rm ada}\eps\leq \theta\eps_{\rm tol}^2/4$.
Moreover, for $\widetilde \rho_\ell<\eps_{\rm tol}$, we may redefine $\widetilde \rho_\ell:= \rho_\ell$.
There holds for $\widetilde \rho_\ell\geq \eps_{\rm tol}$ that
\begin{align*}
 |\rho_\ell^2-\widetilde \rho_\ell^2|\leq C_{\rm ada}\eps \leq \theta \eps_{\rm tol}^2/4\leq 
 \begin{cases}
 \theta \widetilde \rho_\ell^2/4,\\
 2\theta \widetilde \rho_\ell^2/4- \theta C_{\rm ada} \eps \leq \theta\rho_\ell^2/2.
 \end{cases}
\end{align*}
This implies $\widetilde \rho_\ell^2 \geq (1-\theta/2)\rho_\ell^2$ as well as $\rho_\ell^2 \geq (1-\theta/4)\widetilde \rho_\ell^2$.
Assume that $\widetilde \rho_\ell$ satisfies $\sum_{T\in\MM} \widetilde \rho_T^2\geq \theta \widetilde \rho_\ell^2$. Then, the above shows immediately
\begin{align*}
	\sum_{T\in\MM} \rho_{T}^2\geq \sum_{T\in \MM} \widetilde \rho_{T}^2 - \theta \widetilde \rho_\ell^2/4\geq 3\theta/4 \widetilde \rho_\ell^2 \geq 3\theta/8 \rho_\ell^2.
	\end{align*}	
On the other hand, if $\rho_\ell$ satisfies $\sum_{T\in\MM}  \rho_T^2\geq \theta \rho_\ell^2$, there holds
\begin{align*}
	\sum_{T\in\MM} \widetilde \rho_{T}^2\geq \sum_{T\in \MM} \rho_{T}^2 - \theta/2  \rho_\ell^2\geq \theta/2 \rho_\ell^2 \geq 3\theta/8 \widetilde\rho_\ell^2.
	\end{align*}
This equivalence of marking for the two error estimators $\rho$ and $\widetilde \rho$ together with the global equivalence $\widetilde \rho_\ell\simeq \rho_\ell$ allows us to apply~\cite[Theorem~8.4]{axioms} directly to prove optimality of $\widetilde \rho$. This concludes the proof.
\end{proof}
\subsection{Main Result 2}\label{sec:main2}
While the results of Section~\ref{sec:main1} are restricted to second order elliptic PDEs, the following statements deal with a much broader class of problems by making some assumptions on the exact solution.

Let $S:=\bigcup_{k=0}^{d-1}S_k\subset D$ denote a singularity set such that $S_k$ is a finite union of compact $k$-dimensional facets (points for $k=0$, edges for $k=1$, etc). Define the weight
\begin{align*}
 w(x):= \min_{k=0,\ldots,d-1}{\rm dist}(x,S_k)^{\max\{0,d-k-\delta_{\rm reg}\}}
\end{align*}
for some $\delta_{\rm reg}>0$. This induces the weighted space $L^{\infty}_w(D)$ with the norm
\begin{align*}
\norm{v}{ L^{\infty}_w(D)}:=\norm{wv}{L^\infty(D)}.
\end{align*}
Let $u\colon D\to \R$ denote the exact solution of some problem $\LL u = f$ for some operator $\LL\colon H^1_0(D)\to H^{-1}(D)$.
The following result does not depend on the numerical method used  to compute $U_\ell$ and hence we just assume that $\nabla \XX(\TT)\supseteq \PP^0(\TT)$ for all $\TT\in\T$ and that we compute some function $U_\TT\in\XX(\TT)$ by means of some numerical method, i.e., FEM, DG-FEM, \ldots.
Consider the following slight modification of Algorithm~\ref{alg:RNN}:
\begin{algorithm}\label{alg:RNN2}
	\textbf{Input: } Initial mesh $\TT_0$, tolerance ${\eps_{\rm tol}}>0$.\\
	For $\ell=0,1,2,\ldots$ do:
	\begin{enumerate}
		\item Compute discrete approximation $U_{\ell}$.
		\item Apply $\by={\rm ADAPTIVE}(\bx)$ as defined in 
Theorem~\ref{thm:if}.
		\item Use newest-vertex-bisection to refine the elements 
$T_i\in\TT_\ell\setminus\TT_{\ell-1}$ (or $T_i\in\TT_0$ for $\ell=0$) with $y_i> 0$ to obtain a new mesh $\TT_{\ell+1}$ or stop if $\by\leq 0$.
	\end{enumerate}
	\textbf{Output: } Sequence of adaptively refined meshes $\TT_\ell$ and 
corresponding approximations $U_\ell\in\SS^1(\TT_{\ell})$ with $\TT_{\eps_{\rm tol}} := \TT_L$ for final step $L\in\N$.
\end{algorithm}

For the following result, we require a slightly different definition of the maximal
rate:
Let $s>0$ be maximal such that
\begin{align}\label{eq:opt2}
\sup_{N\in\N} \inf_{\TT\in\T\atop \#\TT-\#\TT_0\leq N} \max_{T\in\TT}\inf_{v\in \PP^0(\TT)}\norm{\nabla u - v}{L^2(T)}N^{s+1/2} 
<\infty.
\end{align}
We call Algorithm~\ref{alg:RNN2} optimal if it satisfies 
\begin{align}\label{eq:opt22}
\sup_{0<\eps\leq 1}\#\TT_\eps \Big(\max_{T\in\TT_\eps}\inf_{v\in \PP^0(\TT_\eps)}\norm{\nabla u - v}{L^2(T)}\Big)^{1/(s+1/2)} <\infty.
\end{align}
\begin{remark}
In general, the maximal rate in~\eqref{eq:opt2} is lower than in~\eqref{eq:opt}. However, in many practical situations, the two notions will coincide, as they are equivalent as long as there exists a quasi-best approximating triangulation $\TT$ with $\#\TT-\#\TT_0\leq N$ and 
 \begin{align*}
  \max_{T\in\TT}\inf_{v\in \PP^0(\TT)}\norm{\nabla u - v}{L^2(T)}\lesssim \min_{T\in\TT}\inf_{v\in \PP^0(\TT)}\norm{\nabla u - v}{L^2(T)} + N^{-s-1/2}.
 \end{align*}
This, however, is the case in many approximation results particularly those which include weighted spaces or Besov spaces (see, e.g.,~\cite{approxclass}).
\end{remark}

We denote by $\TT(\eps)\in\T$ the mesh with minimal cardinality such that $
\max_{T\in\TT}\inf_{v\in \PP^0(\TT)}\norm{\nabla u - v}{L^2(T)}\leq \eps$.
\begin{theorem}\label{thm:if}
Let $u\in H^1(D)$ and $m\in\N$.
Suppose there exists a deep RNN $v_\eps$ which satisfies $\norm{\nabla u - v_\eps}{L^2(D)}\leq \eps/(Cm)$
as well as $v_\eps^2 \in L^\infty_w(D)$.
Then, there exists a deep RNN \texttt{ADAPTIVE} such that Algorithm~\ref{alg:RNN2} produces outputs $\TT_\eps$ which satisfy
\begin{align*}
 \#\TT_\eps \Big(\max_{T\in\TT_\eps}\inf_{v\in \PP^0(\TT_\eps)}\norm{\nabla u - v}{L^2(T)}\Big)^{1/(s+1/2)}\leq Cm^{1/(s+1/2)}
\end{align*}
with probability larger than $1-L\#\TT(4\eps) 2^{-Cm}$,
where $C>0$ depends on $D$, $\TT_0$ and $\norm{v_\eps}{L^\infty_w(D)}$ with $L$ denoting the maximal level of elements in $\TT(4\eps)$, i.e., the maximal number of bisections necessary to generate each element from $\TT_0$.
The complexity of ${\rm ADAPTIVE}$ is bounded by
$\mathcal{O}(m^2(\#v_\eps +  |\log(\eps)| + |\log(\norm{v_\eps}{L^\infty(\Dom)})|))$, where $\#v_\eps$ denotes the complexity of $v_\eps$ and the number of independent weights is constant.
\end{theorem}
\begin{remark}
 The dependence of the complexity of ${\rm ADAPTIVE}$ on $\log(\norm{v_\eps}{L^\infty(\Dom)})$ is usually not a problem. Any deep RNN $v_\eps$ approximating $\nabla u$ up to accuracy $\eps>0$ can be capped at magnitude $C>0$ by composition, i.e., $\widetilde v_\eps:=\max(\min(v_\eps,C),-C)$. The approximation error satisfies
 \begin{align*}
  \norm{\widetilde v_\eps -\nabla u}{L^2(\Dom)}\leq \eps + \norm{\nabla u}{L^2(D_C)},
 \end{align*}
with $D_C:=\set{x\in D}{|\nabla u|>C}$. If $\nabla u \in L^p(D)$ for some $p>2$, we already have $|D_C|\leq \norm{\nabla u}{L^p(D)}^p C^{-p}$ and hence $\norm{\nabla u}{L^2(D_C)}\leq \norm{\nabla u}{L^p(D_C)}|D_C|^{(p-2)/(2p)}\lesssim C^{1-p/2}$. This shows that $\log(\norm{v_\eps}{L^\infty(\Dom)})\lesssim |\log(\eps)|$ is possible. 
\end{remark}

We postpone the proof of the above theorem to Section~\ref{sec:if}.

\section{Discussion of the main results}\label{sec:conclude}
\subsection{Theoretical results}
Theorem~\ref{thm:refinement}, Corollary~\ref{cor:refinement}, and Theorem~\ref{thm:if} show that an RNN 
can in fact achieve optimal mesh refinement in the sense of~\eqref{eq:opt} and~\eqref{eq:opt2}.
The RNN only needs to follow the fairly general structure of a deep RNN. The width and depth of the 
RNNs depends poly-logarithmically on the number of elements $\#\TT$ as 
well as on the desired accuracy. The number of independent weights (trainable 
parameters) is, however, uniformly bounded
and independent of the accuracy as well as of the number of elements in the 
mesh.

\bigskip

While Corollary~\ref{cor:refinement} shows that the deep RNN approach is at least as good as current mesh refinement strategies for second order elliptic problems~\eqref{eq:weak} which are known to be optimal, 
Theorem~\ref{thm:if} proves that an optimal mesh refinement strategy can be learned by a deep RNN whenever the exact solution can be approximated efficiently by a deep RNN. The latter result is independent of the problem type and thus applies to problem classes
for which we 
currently do not know optimal refinement strategies. 

Such problems include 
non-linear PDEs (for example~\eqref{eq:weak0} with coefficients depending on $u$). For time dependent PDEs, the current setting based on the $H^1$-norm is too restrictive. However, the proofs can be transferred to any $L^2$-based norm particularly the anisotropic Bochner norms used in parabolic applications. Moreover, the method of proof for Theorem~\ref{thm:if} does not depend on the numerical method used to compute the approximations $U_\ell$. Thus the result also covers non-FEM methods such as discontinuous Galerkin methods, isogeometric analysis methods, boundary element methods and more.

The only requirement is that the exact solution lies in the weighted space $L_w^\infty(D)$ and can be approximated efficiently by a deep RNN (or just a DNN). To that end, we refer to the large number of approximation results for PDEs via neural networks~\cite{dnnGrohsHerrmann, dnnGrohsJentzen,dnnSchwab,dnnBeck,dnnSchwab2} and the references therein. 

\bigskip

If  the data-to-solution map $f\mapsto u$ can be approximated by a deep RNN, then
Theorem~\ref{thm:if} even provides the existence of a deep RNN ${\rm ADAPTIVE}$ which is optimal in the sense~\eqref{eq:opt2} and can be used for any right-hand side data without retraining.

But even if ${\rm ADAPTIVE}$ has to be retrained for each new instance of data, the numerical experiments in Section~\ref{sec:totj} show advantageous performance compared to uniform mesh refinement.

\subsection{Practical implementation}
As stated in~\cite{lstm}, RNNs can be hard to train by gradient 
descent approaches since the recursive
nature either dampens any gradient information or leads to blow-up. The RNNs 
appearing in this work are very sparsely recursive 
(almost all recursive connections are disabled). The existing recursive 
connections on the input sequence $\bx$ 
(and also all intermediate sequences) are always multiplications by 1 or -1 as 
well as additions. Hence those connections do not lead to blowup or dampening.
The constructions include some RNNs with multiplication by 2 or 4 in the 
recursive connections, but those RNNs are always transformed into DNNs and their
size depends only logarithmically on the given accuracy.
Including this observation into the training might improve the performance. 

\bigskip

The training of the deep RNNs can be implemented practically in different ways. For symmetric problems, one may optimize the weights to maximize the energy of the discrete Galerkin approximation (which is equivalent to minimizing the error). This is done in the numerical experiments of Section~\ref{sec:totj}. For more general problems, a substitute energy error is given by
\begin{align*}
 E_\ell:=\sqrt{\sum_{k=\ell}^\infty \norm{U_{k+1}-U_k}{}^2}
\end{align*}
It is shown in~\cite{axioms,stokesopt,fembemopt} that under quite general assumptions, there holds $E_\ell\simeq \norm{u-U_\ell}{}$ up to higher order terms. Thus, to maximize the convergence rate $E_\ell\to 0$, it suffices to maximize $\norm{U_\ell- U_{\ell-1}}{}$ in each adaptive step. Hence, this computable term may serve as a goal quantity for the optimization algorithm.

\section{Construction of the Neural Networks}\label{sec:construction}
This section is dedicated to the construction of the basic building blocks of 
the RNN. 
\subsection{Basic logic \& algebra}
For the implementation of the RNNs below, we require a rudimentary emulation of the ${\rm IF}$-clause. 

\begin{remark}
We note that Assumption~\ref{roundoff} is used particularly in the constructions in this particular section to guarantee that the RNN ${\rm IF}$ constructed below produces the correct output. The RNN ${\rm IF}$ is the sole part of the following constructions, where a round-off error is intentionally scaled to order $\mathcal{O}(1)$. Thus we provide a thorough round-off error analysis in the following Lemma~\ref{lem:if}. In the remaining constructions, ${\rm IF}$ is just used as a building block and we check that input and output of ${\rm IF}$ behave as expected. Thus we follow the usual convention in numerical analysis and do not treat the round-off error explicitly in the calculations outside of ${\rm IF}$.
\end{remark}

%operation. 
%In the following, we use the notation $\R_{\delta+}:= 
%\{0\}\cup\set{x\in\R}{x\geq \delta}$ as well as $\R_\delta:=\{0\}\cup\R\setminus 
%(-\delta,\delta)$. Moreover, we are going to directly exploit the round off errors of floating point arithmetic.
%For simplicity, we restrict to \texttt{double} arithmetic, however, it would be possible to transfer the proofs to any
%floating point system. To that end, we restrict the possible tolerances to $2^\EE$ with $\EE\subset -\N$.
%For \texttt{double} arithmetic, this set is defined by $\EE:=\{-1023+53,\ldots,0\}$.

\begin{lemma}\label{lem:if} 
 For $\square\in\{\leq,\geq,<,>\}$ there exists a fixed size basic RNN ${\rm IF}$ such that any input 
$\bx=( (a,b,c) ,0,\ldots,0)\in\R^{3\times  n}$  with $a,b,c\in\R$ satisfying $|b-c|\geq 2^{-\tilde n}|a|$  results in an output $\by:={\rm IF}(\bx):={\rm IF}(a;b \;\square\; c)\in 
\R^n$ with
\begin{align*}
  y_n=\begin{cases}
          a &b \;\square\; c ,\\
          0 & \textup{else},
         \end{cases}
 \end{align*}
for $n\geq\tilde n$.
If we interpret IF as a DNN, the number of weights behaves like $O(\tilde n^3)$, but the number of independent weights is $O(1)$. With regard to Assumption~\ref{roundoff}, $\tilde n:=n_{\rm min}$ is a valid choice as long as $\max(|c|,|b|)\geq |a|$.
\end{lemma}
\begin{proof} We first define a basic RNN $\widehat{\rm  IF}$ for which, with input $\bx\in \R^{2\times n}$, $\bx=((a,b),0\dots,0)$ with $a\geq0$, the output $\by=\widehat{\rm  IF}(\bx)$ satisfies
\begin{align*}
  y_n=\begin{cases}
          a&b\geq a 2^{-n},\\
          0 & b\leq 0.
         \end{cases}
 \end{align*}
The RNN can be defined by
\begin{align*}
y_i=(y_{i,1},y_{i,2}):=\widehat{\rm IF}(x_i,y_{i-1}):= \big(x_{i,1}+y_{i-1,1},\min(2\max(y_{i-1,2}+x_{i,2},0),y_{i,1})\big).
\end{align*}
 (Note that the first component of $y_i$ has the sole purpose of storing the value of $a$ for later use.) Since $x_i=0$ for all $i\geq 2$ and $y_0=0$, 
we have $ y_i=(a,\min(2^i\max(b,0),a))$.
This concludes the construction of $\widehat{\rm  IF}$.\\
Now let $a,b,c\in\R$ and first assume $a\geq0$. We can see, that ${\rm IF}(a;b>c):= \widehat{\rm IF}(a,b-c)$ produces the expected output, as long as $\tilde n\geq n$ and $|b-c|\geq 2^{-\tilde n}a$: If $b\leq c$, then this is clear, and if $b>c$, it already holds $b-c\geq 2^{-\tilde n}a$ and the first case of $\widehat{\rm IF}$ occurs. 

We can define ${\rm IF}(a;b\leq c)$ by ${\rm IF}(a;b\leq c):= a- {\rm IF}(a;b> c)$, and ${\rm IF}(a;b\geq c),$ ${\rm IF}(a;b< c)$
 can be defined by changing the roles of $b$ and $c$ resulting in the condition $|b-c|\geq 2^{-\tilde n}a$. \\
For $a\in\R$, we set $a_+:=\max(a,0)$ and $a_-:=\max(-a,0)$ and  ${\rm IF}(a;b\;\square\; c):={\rm IF}(a_+;b\;\square\; c)-{\rm IF}(a_-;b\;\square\; c)$ produces the expected output as long as $|b-c|\geq 2^{-\tilde n}\max(a_+,a_-)=2^{-\tilde n}|a|.$
% Under Assumption~\ref{roundoff}, we have for $x>y$ that $x-y=|x-y|\geq 2^{-n_{\rm min}} x$ and therefore ${\rm IF}(x,x-y)$ produces the expected output. Similarly, as long as $|y-z|\geq 2^{-\tilde n_{\rm min}}x $, ${\rm IF}(x,y>z)$ and ${\rm IF}(x,y<z)$ produce the expected output. By ${\rm IF}(x,y\geq z) = x- {\rm IF}(x,y<z)$ and ${\rm IF}(x,y\leq z) = x- {\rm IF}(x,y<z)$ we conclude the proof.
\end{proof} 
\begin{remark}
 Obviously, the RNN $\widehat{{\rm IF}}$ could be constructed as a one layer network $\min(a,2^n \max(b,0))$ at the expense of allowing large weights.
\end{remark}

To emulate the error estimator from Section~\ref{sec:intro}, we require a number 
of basic algebraic operations. We start with squaring. The idea that DNNs can 
emulate the function $x\mapsto x^2$ up to arbitrary precission first appeared 
in~\cite{xsquared}.  They showed that a DNN of size proportional to 
$|\log(\eps)|$ achieves this up to some tolerance $\eps$. We improve on this 
idea by using an RNN of fixed size to perform the same operation. The 
application of the network is equally expensive as the DNN from~\cite{xsquared}, 
however, the number of weights which need to be trained is fixed and independent 
of $\eps$. 

\begin{theorem}\label{thm:square}
For every $n\in\N$, there exists a deep RNN ${\rm SQUARE}$ with a fixed number of 
weights such that the output $\by={\rm SQUARE}(\bx)$ for an input vector 
$\bx=(x_0,0,\ldots,0)\in [-1,1]^n$ satisfies
\begin{align*}
 y_n = {|x_0|}- \sum_{j=1}^n \frac{g^{(j)}({|x_0|})}{4^j}\quad\text{and}\quad 
|y_n-x_0^2|\leq 4^{-2n}.
\end{align*}
for some universal constant $C>0$. If we interpret the basic buildings blocks of the RNN as DNN's, the concatenation of them is still a DNN, so ${\rm SQUARE}$ interpreted as a DNN has a total number of weights of $O(n^3)$, but the number of independent weights stays fixed.
\end{theorem}
\begin{proof}
We reuse the saw-tooth function from~\cite{xsquared}
\begin{align*}
 G(x):=\begin{cases}
        2x & x\in [0,1/2],\\
        2-2x & x\in (1/2,1],
       \end{cases}
\end{align*}
which can also be written as $g(x)=2\max(x,0)-4\max(x-1/2,0)+2\max(x-1,0)$. We define $g:= G\circ G$.
 From the input sequence $\bx\in\R^n$, a first basic RNN layer generates the sequence 
 \begin{align*}
  \bx'=(g(|x_0|),g^{(2)}(|x_0|),\ldots,g^{(n)}(|x_0|)).
 \end{align*}
A second basic RNN $B$ performs the following summation
\begin{align*}
 y_i :=B(x_i,y_{i-1}) = x_i + 4 y_{i-1}.
\end{align*}
This results in
\begin{align*}
 \by=(g(x_0),\ldots, \sum_{j=1}^n 4^{n-j}g^{(j)}(x_0)).
\end{align*}
Finally, the basic RNN $B'$ computes
\begin{align*}
 z_i:=B'(z_{i-1})=z_{i-1}/4.
\end{align*}
Initialized with the last entry $y_n$, this operation computes the vector
\begin{align*}
 \bz = (y_n,y_n/4,\ldots, y_n4^{-n}).
\end{align*}
By definition, $y_n 4^{-n} = \sum_{j=1}^n 4^{-j}g^{(j)}(|x_0|))$. Thus, we 
constructed the desired approximation to $|x_0|-|x_0|^2$.

The error estimate follows from the fact that the approximation to $x^2$ is 
actually the linear spline interpolation $f_n$ of $f(x)=x^2$ at $4^n$ equidistant 
points in $[0,1]$ (see~\cite{xsquared}). This shows
\begin{align*}
 |f_n(x_0)-f(x_0)|\leq \frac{4^{-2n}}{2}\norm{f''}{L^\infty}
\end{align*}
and thus concludes the proof.
\end{proof}

The new idea of the following result is that the magnitude of the input is not 
limited by the number of parameters, but rather by the input length only.
This shows that a fixed number of trainable parameters give a network which can 
multiply arbitrarily large numbers.
\begin{corollary}
 \label{cor:square}
For every $n\in\N$, there exists a {deep} RNN ${\rm SQUARE}$ with a fixed number of 
weights such that the output $\by={\rm SQUARE}(\bx)$ for an input vector 
$\bx=(x_0,0,\ldots,0)\in [-2^{n},2^{n}]^n$ satisfies
\begin{align*}
 y_n = {|x_0|}- \sum_{j=1}^n \frac{g^{(j)}({|x_0|})}{4^j}\quad\text{and}\quad 
|y_n-x_0^2|{\leq 4^{-n}.}
\end{align*}
 {${\rm SQUARE}$ interpreted as a DNN has a total number of weights behaving like $O(n^3)$, but the number of independent weights stays bounded.}
\end{corollary}
\begin{proof}
A first basic RNN performs the scaling
\begin{align*}
 y_i=B^1(x_i):=y_{i-1}/2.
\end{align*}
Note that if $|x_0|\leq 2^n$ there holds $y_n\leq 1$.

We initialize the input of ${\rm SQUARE}$ with the last entry
$y_n$ to compute $z\in\R$ with $|z-y_n^2|\leq 4^{-2n}$.
Finally, we reverse the scaling by initializing a RNN $B^2$ 
with $(z,0,\ldots,0)$ and compute
\begin{align*}
y_i=B^2(x_i):= 4y_{i-1}
\end{align*}
Hence, the final output satisfies
\begin{align*}
 |y_n-x_0^2| = 4^n|z-(x_02^{-n})^2|\leq 4^n 4^{-2n}\leq 4^{-n}.
\end{align*}
This concludes the proof.
\end{proof}

With the squaring operation at hand, we immediately obtain a method for 
multiplying two numbers by using the formula $2xy = (x+y)^2 -x^2-y^2$.

\begin{proposition}\label{prop:multiply}
 There exists a {deep} RNN {\rm MULTIPLY} such that for all $x,y\in [-2^{n-1},2^{n-1}]$ 
the output $\bz={\rm MULTIPLY}(\bx,\by)$ ($\bx,\by\in\R^{n}$ denote the 
sequences $\bx=(x,0,\ldots,0)$, $\by=(y,0,\ldots)$) satisfies
 \begin{align*}
|z_n - xy|\leq C4^{-n},
 \end{align*}
where $C>0$ is independent of $n$ and $x,y$. {${\rm MULTIPLY}$ interpreted as a DNN has a total number of weights behaving like $O(n^3)$, but the number of independent weights stays bounded.}
\end{proposition}

\begin{proof}
 As mentioned above, we construct ${\rm MULTIPY}$ from ${\rm SQUARE}$ with 
 inputs in $[-2^n,2^n]$. The construction is
 \begin{align*}
  {\rm MULTIPLY}(\bx,\by) = ({\rm SQUARE}(\bx+\by) - 
  {\rm SQUARE}(\bx)-{\rm SQUARE}(\by))/2.
 \end{align*}
The error estimate follows immediately from Corollary~\ref{cor:square}.
\end{proof} 

% \begin{remark}\label{rem:round-off}
%  It might seem odd that we actually use the magnification of the round-off error 
% in the construction of the if-clauses in Lemma~\ref{lem:if2}--\ref{lem:if22} but 
% ignore it in Lemma~\ref{prop:multiply}. The reason we chose to do this is that 
% in ${\rm MULTIPLY}$, the magnification of the round-off error only happens 
% in~\eqref{eq:rescale} if the magnitude of the input is large. This means that 
% any rounding error $\eps_{\rm square}$ coming from the application of ${\rm 
% SQUARE}$ gets magnified up to $\simeq \eps_{\rm square}|1+x||1+y|$. This kind of 
% error, however, has to be expected of any arithmetic operation in floating point 
% arithmetic with large inputs and is not a particular flaw of the presented 
% multiplication algorithm (see $\widehat{\rm IF}$ for the contrary case, where 
% small inputs lead to a large, intended round-off error). Hence, to simplify 
% presentation, we believe it is justified to take the usual approach in the 
% numerical analysis of high-level algorithms and to ignore the round-off error as 
% long as scales relatively with the input-output size. The numerical examples in 
% Section~\ref{sec:numerics} underline this argument. 
% \end{remark}

\subsection{Error estimation}

For brevity of presentation, we restrict ourselves to the case $A=1$ 
and $b=c=0$ of~\eqref{eq:weak}. The general case can easily be implemented 
along the lines of this section.
In the present case, the residual error estimator given in~\eqref{eq:errest} is 
usually computed via quadrature. This assumes that $f$ is a piecewise polynomial 
of low enough order such that the quadrature is exact. For convenience, we use 
an equivalent definition of $\rho_T$, i.e.,
\begin{align}\label{eq:equiverr}
 \rho_T(\TT,U_\TT,f)^2\simeq{\rm diam}_{\infty}(T)^{2+d}|T|^{-1}\norm{f+\Delta U_\TT}{L^2(T)}^2
 +{\rm diam}_{\infty}(T)^{d}|\partial T|^{-1}\norm{[\nabla U_\TT]}{L^2(\partial T\cap \Dom)}^2,
\end{align}
with  ${\rm diam}_{\infty}(T):=\max_{x,y\in T}|x-y|_{\infty}$. Obviously, ${\rm 
diam}_{\infty}(T)\simeq {\rm diam}(T)$ depending only on the space dimension.
Moreover, since $ \nabla U_\TT = n\partial_n U_\TT + \sum_{i=1}^{d-1} 
t_i\partial_{t_i} U_\TT$ for normal vector $n$ and tangential vectors 
$t_1,\ldots,t_{d-1}$ and $[\partial_{t_i} U_\TT]=0$ on any interface for $U_\TT\in 
\SS^p(\TT)$, there holds
\begin{align*}
|[\partial_n U_\TT]|^2 = |n[\partial_n U_\TT]|^2=|n[\partial_n 
U_\TT]+\sum_{i=1}^{d-1}t_i[\partial_{t_i} U_\TT]|^2=|[\nabla U_\TT]|^2\quad\text{on }\partial T.
\end{align*}
Note that it would certainly be possible to emulate the exact error estimator 
$\rho(\cdot)$, however, as shown in~\cite{axioms}, a uniform multiplicative 
factor does not make any difference in the convergence behavior and hence we 
opted for the version which results in slightly simpler constructions.

\begin{lemma}\label{lem:diam}
 There is a fixed size DNN ${\rm DIAM}$ which, given the nodes of an element 
$T={\rm conv}(z_0,\ldots,z_d)$ computes the $\infty$-diameter ${\rm 
diam}_{\infty}(T):=\max_{x,y\in T}|x-y|_{\infty}$ of $T$.
\end{lemma}
\begin{proof}
 We exemplify this for $d=2$ and $T={\rm conv}(z_1,z_2,z_3)$, i.e.,
 \begin{align*}
  {\rm diam}_{\infty}(T) = 
\max(\max(|z_1-z_2|_{\infty},|z_1-z_3|_{\infty}),|z_2-z_3|_{\infty}),
 \end{align*}
where $|(x,y)-(x',y')|_{\infty}= \max(|x-x'|,|y-y'|)$ and the absolute value 
function is realized via
\begin{align*}
 |x|=\max(x,0) + \max(-x,0).
\end{align*}
Obviously, this strategy generalizes to higher dimensions.
\end{proof}

\begin{lemma}\label{lem:vol}
Let $U,f\in \PP^p(T)$ for a given element $T\in\TT$.
 There is a {deep} RNN {\rm VOL} which, given the nodes of the element $T={\rm 
conv}(z_0,\ldots,z_d)$ as well as the polynomial coefficients of $U|_T$ and 
$f|_T$ as a  vector valued sequence $\bx=(x,0,\ldots,0)\in\R^{(d(d+1)+ 
2r(p,d))\times n}$, satisfies 
 \begin{align*}
 \big| {\rm diam}_\infty(T)^d|T|^{-1}\norm{f+\Delta U}{L^2(T)}^2-y_n\big|\leq C 2^{-n}
 \end{align*}
for $\by={\rm VOL}(\bx)$ as long as the coefficients of the polynomial expansion 
of $(f+\Delta U)$ and the nodes are contained in $[-2^{\alpha n},2^{\alpha n}]$, where $0<\alpha<1$ depends only on $p$ and $d$. {\rm VOL} interpreted as DNN has a number of weights of $O(n^3)$, but the number of independent weights is fixed.
\end{lemma}
\begin{proof} 
We have $d+1$ points determining the shape of $T$, so $(d+1)d$ scalar numbers, and two input functions in $\PP^p(T)$ with dimension $r(p,d)=\sum_{i=0}^p \binom{d+i-1}{i}$, which makes in total $x\in\R^{ d(d+1)+2r(p,d)}$. 

 Given the polynomial coefficients of $U$, we can compute the coefficients of $\Delta U$ by multiplication with a matrix only depending on $d$ and $p$, which we construct as a DNN.
 Then, we stack $r(d,p)^2$ RNNs ${\rm MULTIPLY}$ from Proposition~\ref{prop:multiply} 
to compute the coefficients of $(f+\Delta U)^2$ up to accuracy $\lesssim 4^{-n}$.  
To compute the integral of the $L^2$-norm, we note that for the basis functions of the polynomial space, $\phi_k(x)= \Pi_{j=1}^d x_j^{\alpha^k_j}$ with exponents $\alpha^k_j$, we have %with $\sum_{j=1}^d \alpha_j\le 2p$
 \begin{align*}
  {\rm diam}_\infty(T)^d|T|^{-1} \int_T \phi_k(x)\,dx = {\rm diam}_\infty(T)^d \int_{\widehat T} \phi_k(F_T(x))\,dx,
 \end{align*}
for the reference element $\widehat T$ and $F_T(x) = (z_1-z_0,\ldots,z_d-z_0)x + 
z_0$. The integral over $\phi_k(F_T(x))$ can be expressed as a sum over integrals over basis functions on the reference element, and from the latter we assume to have them stored in our net as weights, which are scalar numbers only depending on $d$ and $p$. The corresponding coefficients are polynomials of the nodes $z_i-z_0$ and $z_0$, which can be computed by a number of multiplications only depending on $p$ and $d$ with accuracy $\lesssim 4^{-n}$. 
All in all, a number only depending on $d$ and $p$ of instances of ${\rm MULTIPY}$ compute the 
integral with accuracy $\lesssim 4^{-n}$. 
By Proposition \ref{prop:multiply}, all multiplications are computed with the stated tolerance, as long as the coefficients of $f+\Delta U$ and the nodes inserted to an polynomial depending on $d$ and $p$ is contained in $[-2^{n-1},2^{n-1}].$ 
Remark that the above constants still may depend on the magnitude of the input vector. We exemplary consider the computation of a product $\Pi_{i=1}^kx_i,$ (in this situation k depending only on $p$ and $d$). It holds (with $\odot$ denoting the approximate multiplication via ${\rm MULTIPLY}$)
\begin{align*}
\left|\Pi_{i=1}^kx_i- x_1\odot(\dots\odot x_k)\right|&\le x_1\left|\Pi_{i=2}^kx_i- x_2\odot(\dots\odot x_k)\right|+\left|x_1 (x_2\odot(\dots\odot x_k))- x_1\odot(\dots\odot x_k)\right|\\
&\le |x_1|\left|\Pi_{i=2}^kx_i- x_1 (x_2\odot(\dots\odot x_k))\right| +C4^{-n}\\
&\le \dots \le Ck\Pi_{i=1}^k(1+|x_i|) 4^{-n} \le C(k) 2^{-n}. 
\end{align*}
Here we assumed $x_i\lesssim 2^{n/k}$ and the operations $\odot$ are computed with the stated accuracy $\lesssim 4^{-n}$, as the inserted values can be shown to be bounded by $2^{-n-1}$ by induction.
This concludes the proof.
\end{proof}

\begin{lemma}\label{lem:jump}
Let $U,f\in \PP^p(T)$ for a given element $T\in\TT$.
 There is a {deep} RNN {\rm JUMP} which, given the nodes of the elements $T'={\rm 
conv}(z_0,\ldots,z_d)$ as well as the polynomial coefficients of $U|_{T'}$ for 
all elements $T'\in\omega_T$ as a  vector valued sequence 
$\bx=(x,0,\ldots,0)\in\R^{(2(d+1)d +(d+2)r(p,d))\times n}$, satisfies 
 \begin{align*}
 \big|  {\rm diam}_\infty(T)^{d-1}|\partial T|^{-1}\norm{[\nabla U]}{L^2(\partial T)}^2-y_n\big|\leq C2^{-n}
 \end{align*}
for $\by={\rm JUMP}(\bx)$ as long as the coefficients of the polynomial 
expansion of $[\nabla U]$ and as long as the coefficients of the polynomial expansion of $[\nabla U]$ are contained in $[-2^{\alpha n},2^{\alpha n}]$, where $0<\alpha<1$ depends only on $p$ and $d$. {\rm JUMP} interpreted as DNN has a number of weights behaving like $O(n^3)$, but the number of independent weights is fixed. 
\end{lemma}
\begin{proof}
As input, we have $d+1$ nodes determining the shape of T, and another $d+1$ elements in the patch which are determined by another $d+1$ points. So $(2d+2)d$ scalar variables for the nodes and $(d+2)r(p,d)$ for the polynomial coefficients of $U$, which results in $x\in\R^{2(d+1)d +(d+2)r(p,d)}$.
The proof works analogously to that of Lemma~\ref{lem:vol}, with the difference 
that we have to include the data on the patch of $T$ to compute $[\nabla U]$.
\end{proof}

\begin{theorem}\label{thm:estimate}
There exists a {basic} RNN ${\rm ESTIMATOR}$ which takes a vector-valued input sequence 
$\bx\in\R^{(2(d+1)d +(d+3)r(p,d))\times  \#\TT}$ such that $x_i$ contains: The nodes 
of the elements $T'\in\omega_{T_i}$ for $T_i\in\TT$ and the corresponding 
polynomial expansions of $U_{T'}$ and $f|_{T_i}$. The output $\by:= {\rm 
ESTIMATOR}(\bx)$ satisfies
\begin{align*}
 |y_i-\rho_{T_i}(\TT,U,f)^2|\leq C 2^{-n}
\end{align*}
in case $n\gtrsim \max(\log(\bx))$ for a uniform hidden constant. The RNN {\rm 
ESTIMATOR} has a fixed number of independent weights but width and depth proportional to 
$n$, so a total number of weights behaving like $O(n^3)$. 

\end{theorem}
\begin{proof}
 Lemmas~\ref{lem:diam}--\ref{lem:jump} show that there are RNNs computing all 
ingredients for $\rho_T(\TT,u,f)^2$. A fixed number of applications of the RNN {\rm MULTIPLY} 
combine the elements and output an approximation to $\rho_T(\TT,u,f)^2$ up to an 
 accuracy $\lesssim 2^{-n}$ as long as the magnitude of the input is bounded 
by $2^{\alpha n}$ (where $n\in\N$ is the size of the input sequence and $\alpha$ only depends on $d$ and $p$). Interpreting 
the resulting RNN ${\rm EST}$ which computes the approximation to 
$\rho_T(\TT,u,f)^2$ as a DNN, we observe that ${\rm EST}$ is a DNN with depth 
and width $\mathcal{O}(n)$ composed of $n$ copies of the same net. Hence, we 
only have a fixed (accuracy independent) number of independent weights in ${\rm EST}$ 
although the width and depth of ${\rm EST}$ depends on $n$. Moreover, ${\rm EST}$ forms 
the building block for an RNN which takes a vector-valued sequence 
$\bx\in\R^{({2(d+1)d +(d+3)r(p,d)})\times  \#\TT}$  as described in the statement.
 From this, ${\rm EST}$ computes the output sequence $y_i$ which satisfies
 \begin{align*}
  |y_i-\rho_{T_i}(\TT,U,f)^2|\lesssim 2^{-n}.
 \end{align*}
A final application $\by=\max(\by,0)$ guarantees the non-negativity of the estimators and this concludes the proof.
\end{proof}
\subsection{D\"orfler marking}
The marking algorithm is based on the following observation: Assume $x_1,\ldots,x_n\geq 0$. Consider the binary search algorithm
\begin{algorithm}\label{alg:binary}
 \textbf{Input:} $x_1,\ldots,x_n\geq 0$, $0<\theta\leq 1$\\
 Set ${y_1}:= \max_{1\leq i\leq n}x_i/2$. For $\ell={1},\ldots, k $ do:
 \begin{enumerate}
  \item If $\sum_{x_i\geq y_\ell} x_i \geq \theta \sum_{i=1}^n x_i$, set{ $y_{\ell+1} = y_\ell+y_1/2^\ell .$}
   \item If $\sum_{x_i\geq y_\ell} x_i< \theta \sum_{i=1}^n x_i$, set  {$y_{\ell+1}=y_\ell-y_1/2^\ell$.}
    %\item If $\sum_{x_i\geq y_\ell} x_i= \theta \sum_{i=1}^n x_i$, set $y_{\ell+1} = y_\ell$.
 \end{enumerate}

\end{algorithm}

\begin{lemma}\label{lem:binary}
 Assume $x_1,\ldots,x_n\geq 0$. Let $y\geq 0$ be maximal such that $\sum_{x_i\geq y } x_i\geq \theta \sum_{i=1}^n x_i$. Then, there holds $|y-y_\ell|\leq 2^{-\ell}\max_{1\leq i\leq n}x_i$.
\end{lemma}
\begin{proof}
 The binary search nature of the algorithm immediately guarantees $|y-y_\ell|\leq 2^{-\ell}\max_{1\leq i\leq n}x_i$.
\end{proof}

\begin{theorem}\label{thm:binary}
There exists a deep RNN ${\rm BINARY}$ consisting of $k$ basic RNN's of the same type (up to an fixed size input layer), which takes as input the sequence $\bx=(x_1,\ldots,x_n)\in\R^{1\times n}$ and the output $\by:={\rm BINARY}(\bx)\in\R$ satisfies
$$|y-y_k|\leq 2^{-k} \max_{1\leq i\leq n}x_i,$$
 where $y\geq 0$ is maximal such that
$\sum_{x_i\geq y} x_i \geq \theta \sum_{i=1}^n x_i$. The number of weights of the basic RNNs is bounded by $O(n_{\rm min}^3)$, the number of independent weights is fixed and as the deep RNN consists of $k$ copies of the same basic RNN, the total number of weights is bounded by $O(kn_{\rm min}^3)$, while the number of independent weights stays bounded independently.
\end{theorem}
\begin{proof}
%Let $\delta=\min_{x_i\neq x_j}|x_i-x_j|>0$. By Assumption~\ref{roundoff}, there holds $\delta\geq 2^{-n_{\rm min}} \min_{1\leq  i\leq n}x_i$.

We may construct a basic RNN ${\rm SUMY}$ with inputs $(x_1,\ldots,x_n)$ and $y$ via
\begin{align*}
 z_i={\rm SUMY}(x_i,y,z_{i-1}):= (z_{i-1} +{\rm IF}(x_i;x_i\geq y)) %+ x_i-{\rm IF}(y-x_i,x_i)).
\end{align*}
Lemma~\ref{lem:if} shows that ${\rm IF}$ computes the exact cut-off function, as we assume $x_i$ to satisfy Assumption~\ref{roundoff}. This shows
\begin{align*}
 z_n= \sum_{x_i\geq y} x_i
\end{align*}
and the number of weights of ${\rm SUMY}$ behaves like $O(n_{\rm min}^3)$.
A basic RNN gives the initial values $y_1:= \max_{1\leq i\leq n}x_i/2$, $z_1:=y_1/2$ and we construct the basic RNN that performs one iteration of Algorithm~\ref{alg:binary}.  One iteration of Algorithm~\ref{alg:binary} corresponds to
\begin{align*}
 \binom{\by_{\ell+1}}{\bz_{\ell+1}}&=\binom{\by_\ell+ {\rm IF}(\bz_\ell ; {\rm SUMY}(\bx,y_{\ell,n})\geq\theta {\rm SUMY}(\bx,0))
 %{\rm IF}({\rm SUMY}(\bx,y_\ell)-\theta {\rm SUMY}(\bx,0),y_\ell/2) \\
 - {\rm IF}(\bz_\ell ; {\rm SUMY}(\bx,y_{\ell,n})<\theta {\rm SUMY}(\bx,0))}{\bz_\ell/2 }. 
  %+{\rm IF}(-{\rm SUMY}(\bx,y_\ell)+\theta {\rm SUMY}(\bx,0),(\max_{1\leq i\leq n}x_i-y_\ell)/2).
\end{align*}
It is $\bz_\ell= \by_1/2^\ell$ and $y_{\ell,n}$ (which is the last entry of $\by_\ell$) contains the current pivot. We assume that ${\rm SUMY}(\bx,y_{\ell,n})$, $\theta {\rm SUMY}(\bx,0)$ satisfy Assumption~\ref{roundoff} and since ${\rm SUMY}(\bx,y_\ell)\geq \max_{1\leq i\leq n}x_i\geq \bz_\ell $,
%since ${\rm SUMY}(\bx,y_\ell)\geq y_\ell$ and $|{\rm SUMY}(\bx,y_\ell)-\theta {\rm SUMY}(\bx,0)|\geq 2^{-n_{\rm min}} {\rm SUMY}(\bx,y_\ell)$ by Assumption~\ref{roundoff}, 
the cut-off functions are computed exactly. % as long as the accuracy of ${\rm IF}$ is sufficiently large depending on $n_{\rm min}$. 
Remark that one mapping $(y_l,z_l)\mapsto (y_{l+1},z_{l+1})$ corresponds to an application of a DNN of size $O(n_{\rm min}^3)$ to ${\rm SUMY}(\bx,y_\ell)$ and ${\rm SUMY}(\bx,0)$. 
This can be
constructed as follows (see also Figure~\ref{fig:RNNBinary}). A basic RNN of comparable size to ${\rm SUMY}$ takes the input $(\bx, y_{\ell,n})$ and computes the
output $\widetilde \bx$ with $\widetilde x_i
:= ({\rm SUMY}(\bx, y_{\ell,n}), {\rm SUMY}(\bx, 0))$ for all $1 \leq  i \leq  n$. Then a basic RNN containing the DNN
IF is applied to $\widetilde \bx$ to compute the output $\by_{\ell+1}$ and $\bz_{\ell+1}$. This concludes the proof.

\end{proof}

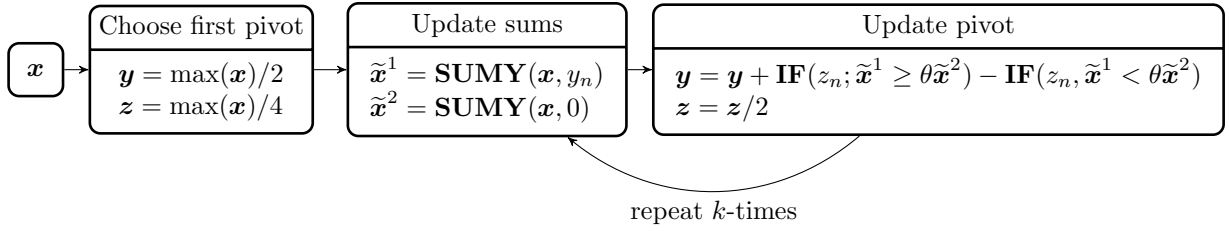
\begin{figure}

	\begin{tikzpicture}[->,>=stealth']
	
	% Position of QUERY 
	% Use previously defined 'state' as layout (see above)
	% use tabular for content to get columns/rows
	% parbox to limit width of the listing

	% STATE QUERYREP
	\node[state] (INPUT) 
	{%
		\begin{tabular}{l}
		$\bx$
		\end{tabular}
	};
	\node[block,right of=INPUT,xshift = 1.2cm] (BINARY){Choose first pivot\nodepart{two} 
	\begin{tabular}{l}
		$\by=\max(\bx)/2$\\ 
		$\bz=\max(\bx)/4$
		\end{tabular}};
	\node[block,right of=BINARY,xshift = 2.8cm] (ROUND) {Update sums\nodepart{two}
	\begin{tabular}{l}
		$\widetilde\bx^1 = {\bf SUMY}(\bx,y_n)$\\
		$\widetilde \bx^2 = {\bf SUMY}(\bx,0)$
		\end{tabular}};
	\node[block,right of=ROUND,xshift = 5cm] (TMP1) {Update pivot\nodepart{two}\begin{tabular}{l}
		$\by=\by+{\bf IF}(z_n;\widetilde\bx^1\geq \theta\widetilde\bx^2)-{\bf IF}(z_n,\widetilde\bx^1<\theta \widetilde\bx^2)$\\
		$\bz = \bz/2$
		\end{tabular}};
	
	% draw the paths and and print some Text below/above the graph
	\path (INPUT) edge  node[anchor=south,above]{} (BINARY)
	(BINARY) edge  node[anchor=south,above]{} (ROUND)
	(ROUND) edge  node[anchor=south,above]{} (TMP1)
	(TMP1) edge[bend left =40] node[anchor=south,below]{repeat $k$-times} (ROUND)
	;	%node[anchor=north,below]{$RN_{16}$} (ACK)
	%(INPUT);%     	edge[bend right=20] node[anchor=south,above]{$SC_n\neq 0$} (NONZERO);
	%(ACK)       	edge                                                     
%(EPC)
	%(EPC)       	edge[bend left]                                          
%(QUERYREP)
	%(QUERYREP)  	edge[loop below]    node[anchor=north,below]{$SC_n\neq 
%0$} (QUERYREP)
	%(QUERYREP)  	edge                node[anchor=left,right]{$SC_n = 0$} 
%(ACK);
	
	\end{tikzpicture}
	\caption{The structure of the RNN ${\rm BINARY}$ from Theorem~\ref{thm:mark}. Variables which are not used in a particular block are copied to the output sequence. The RNN $\by=\max(\bx)$ is defined by $y_{i}=\max(y_{i-1},x_i)$ and computes $y_i=\max_{1\leq j\leq i}x_j$.}
	\label{fig:RNNBinary}
	
\end{figure}
%\begin{lemma}\label{lem:round}
% There exists a basic RNN ${\rm ROUND}$ which takes as input a sequence $x_1,\ldots,x_n$ as well as a number $y$ and outputs the number $x_i$ such that $|x_i-y|$ is minimal.
%\end{lemma}
%\begin{proof}
% We construct the RNN by first copying $y$ to the whole sequence and then feed the sequence
% \begin{align*}
% \bx:= \begin{pmatrix}
%   x_1 &\ldots & x_n\\
%   y &\ldots & y
%  \end{pmatrix}
% \end{align*}
%into
% \begin{align*}
%  z_i = {\rm ROUND}(x_i,y_i):= {\rm IF}(x_i; |z_{i-1}-y_i|>|x_i-y_i|) + {\rm IF}(z_{i-1};|z_{i-1}-y_i|\leq|x_i-y_i|).
% \end{align*}
%Under Assumption~\ref{roundoff}, there holds $||z_{i-1}-y_i|-|x_i-y_i||\geq 2^{-2n_{\rm min}} \max(x_i,z_{i-1})$ or $|z_{i-1}-y_i|-|x_i-y_i|=0$. Hence, ${\rm IF}$ with sufficient accuracy depending only on $n_{\rm min}$ computes the correct cut-off function. Remark that the RNN can be written as basic RNN, by leaving out the copying step and therefore using an additional output variable in the second RNN which stores $y$. This concludes the proof. 
%\end{proof}

\begin{lemma}\label{lem:round}
There exists a basic RNN ${\rm ROUND}$ which takes as input a non-negative sequence $\bx=(x_1,\ldots,x_n)\in \R^{1\times n}$ and values $\overline{x}>\underline{x}\geq0$ and outputs a sequence $\by=(y_1,\ldots,y_n):={\rm ROUND}(\bx)\in\R^{1\times n}$ which satisfies
 \begin{align*}
   y_i= \begin{cases}
\overline{x} , \quad x_i\in[ \underline x,\overline x]\\
x_i, \quad \textup{else}.
\end{cases}
 \end{align*}
 The number of weights behaves like $O(n_{\rm min}^3)$, while the number of independent weights stays bounded. 
\end{lemma}
\begin{proof}
The RNN can be constructed as the component-wise maximum of the sequences $\bx$ and $\overline \bx$, where 
\begin{align*}
\overline x_i:= \begin{cases}
\overline{x} , \quad x_i\in[\underline x,\overline x]\\
0, \quad \textup{else},
\end{cases} 
= \overline{ x} - {\rm IF} (\overline{x}; x_i>\overline{x}) - {\rm IF} (\overline{x}; x_i+\overline{x}<\underline{x}+\overline{x}).
\end{align*}
The ${\rm IF}$ are interpreted as DNN's of size $O(n_{\rm min}^3)$, and compute the expected output, as we assume $x_i,$ $\overline{x}$, $\underline{x}$ to satisfy Assumption \ref{roundoff} and it holds $\max(|x_i+\overline{x}|,|\underline{x}+\overline{x}|\geq |\underline{x}+\overline{x}|) \geq \overline{x}$. 
\end{proof}
%\begin{remark}
%If one does not feel comfortable with the insertion of $\overline{x}$, one should reformulate this lemma with an condition $\overline{x}\le C\underline{x}$, and this can then be applied if we add on the whole sequence $z_l$ beforehand...
%\end{remark}

\begin{theorem}\label{thm:mark}
 There exists a deep RNN ${\rm MARK}$ which takes as input a non-negative sequence $\bx=(x_1,\ldots,x_n)\in \R^{1\times n},$ and outputs a sequence $\by=(y_1,\ldots,y_n):={\rm MARK}(\bx)\in\R^{1\times n}$ which satisfies
 \begin{align*}
  \sum_{i=1\atop y_i> 0}^n \tilde x_i\geq \theta \sum_{i=1}^n \tilde x_i
 \end{align*}
such that the number of terms in the left-hand side sum is minimal and  $\tilde \bx$ satisfies $|x_i-\tilde x_i|\le \eps/n.$. The deep RNN consists of a fixed layer of basic RNN's, followed by $ k\simeq \log_2(\max_{1\leq i\leq n}x_i) + |\log_2(\eps/n)|+1$ copies of the same basic RNN and ends with an output layer of a fixed number of basic RNN's. The overall number of weights therefore behaves like $O(k n_{\rm min}^3),$ while the number of independent weights stays bounded.
\end{theorem}
\begin{proof}
See also Figure~\ref{fig:RNNMark} for the following construction:
As a first layer, we have the deep RNN ${\rm BINARY}$ with $k\simeq  \log_2(\max_{1\leq i\leq n}x_i) +  |\log_2(\eps/n)|+1$ repetitions. This produces the pivot $y_k$ from Algorithm~\ref{alg:binary} with  $|y_k-y|\le 2z_k\le \eps_{\rm tol}/(2n)$ and with $y\geq 0$ maximal such that $\sum_{x_i\geq y } x_i \geq  \theta \sum_{i=1}^n x_i$. Now, for $z_k$ from ${\rm BINARY}$, it holds $y\in[y_k-2z_k,y_k+2z_k]$ and an application of ${\rm ROUND}$ gives the sequence $\tilde \bx$ for $\underline x:=y_k-2z_k$, $\overline x:=y_k+2z_k$ with the stated error bound, as $2z_k\le \eps/(2n)$. %Or: This already shows, as $y\geq \eps/n$, that $y=y_k$ due to Assumption~\ref{roundoff}. ? dann reicht oben +1? $y=x_i$ for an $i$, so $y$ satisfies Assumption~\ref{roundoff}, and also $y_k$ satisfies it, because we assume that?  
We now set $\overline y:= \overline x$ and it holds that $\overline y$ is maximal such that $\sum_{\tilde x_i\geq \overline y} \tilde x_i \geq \theta \sum_{i=1}^n \tilde x_i.$ This is because it holds 
\begin{align*}
\sum_{\tilde x_i>\overline y}\tilde x_i= \sum_{\tilde x_i>\overline y}  x_i \le \sum_{x_i > y} x_i <\theta \sum x_i\le \theta\sum{\tilde x_i}
\end{align*}
and 
\begin{align*}
\sum_{\tilde x_i\geq\overline y}\tilde x_i = \sum_{\tilde x_i\geq\overline y} x_i + \sum_{\tilde x_i=\overline y} (\tilde x_i-x_i) \geq \sum_{ x_i\geq  y} x_i + \theta \sum_{\tilde x_i=\overline y} (\tilde x_i-x_i) \geq \theta \sum x_i + \theta \sum_{\tilde x_i=\overline y} (\tilde x_i-x_i) = \theta \sum \tilde x_i,
\end{align*}
where we used that $x_i\geq y$ implies $\tilde x_i\geq \overline{y}$.
Hence we found the exact cutoff $\overline y$ for the sequence $\tilde \bx$ and proceed with this new sequence.
We generate a preliminary output sequence $\widetilde y_i:=\max(\tilde x_i-\overline y,0)$. The final output $y$ is positive whenever $\widetilde y_i>0$ and additionally on a few entries with $\tilde x_i=\overline y$. To find those entries, compute the sequence
 \begin{align*}
 \widehat x_i = \overline y-{\rm IF}(\overline y;\overline y>\tilde x_i)-{\rm IF}(\overline y;\overline y<\tilde x_i)
\end{align*}
such that $(\widehat x_1,\ldots,\widehat x_n)$ is zero
unless $\widehat x_i=\tilde x_i=\overline y$. The computations of ${\rm IF}$ are exact as before, and the complexity of this basic RNN is of $O(n_{{\rm min}}^3)$ with $n$ repetitions.
Next, we compute
 \begin{align*}
 z_i:= {\rm IF}(\overline y; \sum_{j=1}^{i-1}\widehat x_j+\sum_{\tilde x_i>\overline y} \tilde x_i <\theta \sum_{i=1}^n \tilde x_i)
 \end{align*}
There holds $z_i=\overline y$ for all $1\leq i\leq i_0$ and $z_i=0$ else for minimal $i_0$ such that $\sum_{j=1}^{i_0}\widehat x_j+\sum_{\tilde x_i>\overline y} \tilde x_i \geq\theta \sum_{i=1}^n \tilde x_i$. Note that $\sum_{j=1}^{i-1}\widehat x_j$ can be computed beforehand by a basic RNN of size $\mathcal{O}(n_{\rm min}^3)$ as a straightforward modification of ${\rm SUMY}$ in Theorem~\ref{thm:binary}. 
As usual, we use Assumption~\ref{roundoff} to guarantee that the cut-off functions are computed exactly, as we assume the sums to satisfy Assumption~\ref{roundoff}. 
Finally, we generate the desired output with
\begin{align*}
 y_i=\max(\widetilde y_i, \min(z_i,\widehat x_i)).
\end{align*}
This concludes the proof.
\end{proof}

\begin{figure}
	\begin{tikzpicture}[->,>=stealth']
	
	% Position of QUERY 
	% Use previously defined 'state' as layout (see above)
	% use tabular for content to get columns/rows
	% parbox to limit width of the listing

	% STATE QUERYREP
	\node[state] (INPUT) 
	{%
		\begin{tabular}{l}
		$\bx$
		\end{tabular}
	};
	\node[block,right of=INPUT,xshift = 2cm] (BINARY) {Compute pivot $y_k$\nodepart{two}
	\begin{tabular}{l}
		$(y,z)={\bf BINARY}(\bx)$\\
		$\overline{y}=y+2z$\\
		$\underline{y}=y-2z$
		\end{tabular}};
	\node[state,right of=BINARY,xshift = 3.5cm] (ROUND) 
	{\begin{tabular}{l}
		$\widetilde \bx = {\bf ROUND}(\bx,\overline{y},\underline{y})$
		\end{tabular}};
	\node[block,right of=ROUND,xshift = 4.5cm] (TMP1){ Store entries equal to cut-off in $\widehat{\bx}$ \nodepart{two}
	\begin{tabular}{l}
		$\widetilde \by = \max(\widetilde \bx-\overline{y},0)$\\
		$\widehat\bx =\overline{y}-{\bf IF}(\overline{y};\overline{y}>\widetilde \bx)-{\bf IF}(\overline{y};\overline{y}<\widetilde \bx)$
		\end{tabular}};
	\node[block,below of=TMP1,yshift = -1.5cm] (TMP2){ Compute $\sum_{\widetilde x_i>\overline{y}} \widetilde x_i$ and $\sum_{j=0}^{i-1}\widehat x_j$ \nodepart{two}
	\begin{tabular}{l}
		$\widetilde{\boldsymbol{S}}=\widetilde {\rm SUMY}(\widetilde \bx,\overline{y})$\\
		${\boldsymbol{S}}={\rm SUMY}(\widehat \bx,0)-\widehat{\bx}$
		\end{tabular}};
	\node[block,left of=TMP2,xshift = -5cm] (CHOOSE1) {Find minimal $i_0$\nodepart{two}
	\begin{tabular}{l}
		$\bz = {\bf IF}(\overline{y}; \widetilde{S_n}+ \boldsymbol{S}< \theta S_n)$
		\end{tabular}};
	\node[state,left of=CHOOSE1,xshift = -4cm] (OUTPUT) 
	{\begin{tabular}{l}
		$\by = \max(\widetilde \by,\min(\bz,\widehat\bx))$
		\end{tabular}};	
	% draw the paths and and print some Text below/above the graph
	\path (INPUT) edge  node[anchor=south,above]{} (BINARY)
	(BINARY) edge  node[anchor=south,above]{} (ROUND)
	(ROUND) edge  node[anchor=south,above]{} (TMP1)
	(TMP1) edge  node[anchor=south,above]{} (TMP2)
	(TMP2) edge  node[anchor=south,above]{} (CHOOSE1)
	(CHOOSE1) edge  node[anchor=south,above]{} (OUTPUT)
	;	%node[anchor=north,below]{$RN_{16}$} (ACK)
	%(INPUT);%     	edge[bend right=20] node[anchor=south,above]{$SC_n\neq 0$} (NONZERO);
	%(ACK)       	edge                                                     
%(EPC)
	%(EPC)       	edge[bend left]                                          
%(QUERYREP)
	%(QUERYREP)  	edge[loop below]    node[anchor=north,below]{$SC_n\neq 
%0$} (QUERYREP)
	%(QUERYREP)  	edge                node[anchor=left,right]{$SC_n = 0$} 
%(ACK);
	
	\end{tikzpicture}
	\caption{The structure of the RNN ${\rm MARK}$ from Theorem~\ref{thm:mark}. Variables which are not used in a particular block are copied to the output sequence. Non-bold variables are implemented as constant sequences. The RNN $\widetilde{\rm SUMY}$ is a straightforward modification of ${\rm SUMY}$ from Theorem~\ref{thm:binary} by replacing $\geq$ with $>$.}
	\label{fig:RNNMark}
	
\end{figure}
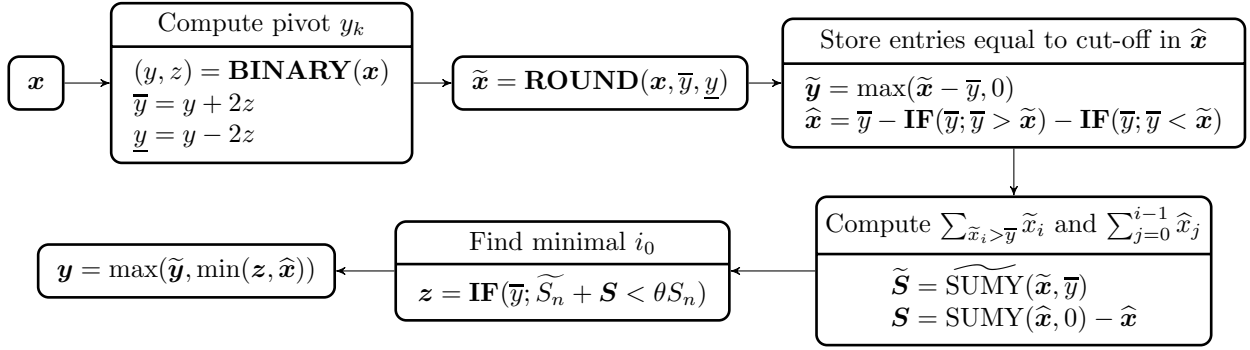

\subsection{Proof of Theorem~\ref{thm:refinement}}\label{sec:complete}
The previous sections already give the necessary ingredients to build the RNN 
{\rm ADAPTIVE} from Theorem~\ref{thm:refinement}.
We use the RNN {\rm ESTIMATOR} with accuracy $n\simeq |\log(\eps/N)|$ from Theorem~\ref{thm:estimate} to compute the error estimator 
$\widetilde \rho_T(\TT,U_\TT,f)$ such that
\begin{align}\label{eq:approxest}
 |\widetilde \rho_T(\TT,U_\TT,f)^2-\rho_T(\TT,U_\TT,f)^2|\lesssim \eps/\#\TT
\end{align}
for all $T\in\TT$ as long as $N\geq \#\TT$ and $|\log(\eps/N)|\gtrsim \log(|x|_\infty)$. This results in a number of weights of $O(|\log(\eps/N)|^3)$.\\ 
Theorem~\ref{thm:mark} provides the deep RNN ${\rm MARK}$ with $n=\#\TT$, which performs the D\"orfler marking and adds an additional error of $\eps/\#\TT$ to the error estimators. 
%as long as $y\geq \eps/\#\TT$. If $y<\eps/\#\TT,$ another application of ${\rm IF}$ (which is not pointed out in Figure 1) can filter out error estimators smaller than the $y$ used in the RNN, until the D\"orfler marking criterion is fulfilled again, which only adds an additional error of $\eps/\#\TT$ to elements which are close to $0$.} In both cases, 
This results in a number of weights of 
\begin{align*}
O\Big(\big(\log_2(\max_{1\leq i\leq \#\TT}\tilde\rho_{T_i}) + |\log_2(\eps/\#\TT)|\big) n_{\rm min}^3\Big).
\end{align*}
All the basic building blocks used in the construction consist of a fixed number of 
independent weights but may have a total number of weights depending on $\log(N)$, $n_{\rm min}^3$ 
and $\log(\eps)$. One of the basic blocks is stacked $\left(\log_2(\max_{1\leq i\leq \#\TT}\tilde\rho_{T_i}) + |\log_2(\eps/\#\TT)|\right)$-times. We ensure the stopping criterion of the algorithm by comparing the sum of the error estimators with the tolerance $\eps_{\rm tol}$. The full structure of ${\rm ADAPTIVE}$ is given in Figure~\ref{fig:RNNADAPTIVE}.

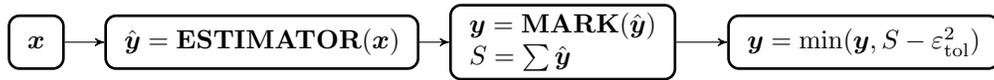
\begin{figure}[h]
	\begin{tikzpicture}[->,>=stealth']	
	
	\node[state] (INPUT) 
	{%
		\begin{tabular}{l}
		$\bx$\\
		\end{tabular}
	};
	\node[state,right of=INPUT,xshift = 2cm] (ESTIMATE) 
	{\begin{tabular}{l}
		$\hat\by = {\bf ESTIMATOR}(\bx)$
		\end{tabular}};
	\node[state,right of=ESTIMATE,xshift = 3cm] (CUT) 
	{\begin{tabular}{l}
		$\by={\bf MARK}(\hat\by)$\\
		$S=\sum \hat\by$
		\end{tabular}};
	\node[state,right of=CUT,xshift = 3cm] (OUTPUT) 
	{\begin{tabular}{l}
		$\by=\min(\by, S-\eps^2_{\rm tol})$
		\end{tabular}};
	\path (INPUT) edge  node[anchor=south,above]{} (ESTIMATE)
	(ESTIMATE) edge node[anchor=south,above]{} (CUT)
	(CUT) edge node[anchor=south,below right]{} (OUTPUT);
	
	\end{tikzpicture}
	\caption{Structure of the RNN {\rm ADAPTIVE}. For a given tolerance $\eps_{\rm tol}^2>0$, the term $S-\eps_{\rm tol}^2$ is non-positive whenever the prescribed tolerance has been reached by the approximate error estimator stored in $S$ and hence terminates the algorithm by setting $\by\leq 0$. }
	\label{fig:RNNADAPTIVE}
	
\end{figure}

\subsection{Proof of Theorem~\ref{thm:if}}\label{sec:if}
The proof of Theorem~\ref{thm:if} requires some preparation. The first result states that for functions in certain weighted $L^\infty$ spaces, the Monte Carlo method overestimates the integral with a positive probability. 
\begin{lemma}\label{lem:EX}
For a given Lipschitz domain $\Omega$ (not necessarily $D$) and a singularity set $S$, let $X\colon \Omega \to \R$ be non-negative with $X\in L^\infty_w$.
%and $p>1$ such that $p/(p-1)<\min_{k=0,\ldots,d-1}\frac{d-k}{d-k-\delta_{\rm reg}}$.
We assume $|\Omega|=1$, and consider $\Omega$ as a probability space and $X$ as a random variable.
There holds
\begin{align*}
 \P(X\geq q\E(X))\geq C^{-1}
 \end{align*}
for all $0<q<1$ and a constant $C>0$ which depends on $q$, $\Omega$, the number of connected components in $S_k$, $k=0,\ldots,d-1$,  $\delta_{\rm reg}$, and an upper bound for $\norm{X}{L^\infty_w(\Omega)}$.
\end{lemma}
\begin{proof}
Let $p>1$ such that $p/(p-1)<\min_{k=0,\ldots,d-1}\frac{d-k}{d-k-\delta_{\rm reg}}$. Define $\Omega_\geq :=\set{\omega\in\Omega}{X(\omega)\geq q\E(X)}$ and observe
\begin{align*}
 \E(X)&=\int_{\Omega\setminus\Omega_\geq} X\,d\omega + \int_{\Omega_\geq} X\,d\omega\leq q\E(X) + \norm{w^ {-1}\id_{\Omega_\geq}}{L^{p/(p-1)}(\Omega)}\norm{wX}{L^p(\Omega)},
\end{align*}
where $w$ is the weight of $L^\infty_w$ defined in Section~\ref{sec:main2}.
It remains to estimate $\norm{w^{-1}\id_{\Omega_\geq}}{L^{p/(p-1)}(\Omega)}$. We aim to prove $\norm{w^{-1}}{L^{r}(\Omega)}<\infty$ for $r>1$ such that $r(d-k-\delta_{\rm reg})<d-k$ with $k=0,\ldots,d-1$. To that end, we bound the integrand $|w|^{-r}$ from above by functions of the form $x\mapsto {\rm dist}(E,x)^{-\alpha}$, with $0<\alpha<d-k$ and $E$ denoting a facet of dimension $k$. All functions of this type are in $L^1(\Omega)$ and we obtain
\begin{align*}
  \norm{w^ {-1}\id_{\Omega_\geq}}{L^{p/(p-1)}(\Omega)}\leq  \norm{w^ {-1}}{L^{r}(\Omega)}\norm{\id_{\Omega_\geq}}{L^{r'}(\Omega)}^{(p-1)/p}\lesssim |\Omega_\geq|^{p/(r'(p-1))},
\end{align*}
where $r'=r/(r-p/(p-1))$.
A H\"older inequality shows
\begin{align*}
 \norm{w X}{L^p(\Omega)}^p\leq \E(X)\norm{w^p X^{p-1}}{L^\infty(\Omega)}\lesssim \E(X)\norm{X}{L^\infty_{w}(\Omega)}^{p-1}
\end{align*}
and hence concludes the proof.

\end{proof}
Let $\TT_\infty:=\bigcup_{\TT\in\T}\TT$ denote the set of all possible elements which may appear as refinements of some elements in $\TT_0$.
With an error estimator $\eta(T,V)$, which for now is just a function depending on $T\in\TT_\infty$ and $V\in L^2(\Dom)$,
we base the construction of our deep RNN ${\rm ADAPTIVE}$ on the following greedy algorithm:
\begin{algorithm}\label{alg:greedy}
 Input: Function $V\in H^1(D)$, tolerance $\eps>0$, initial mesh $\TT_0$\\
 For $\ell = 0,1,2,\ldots$ do:
 \begin{itemize}
  \item[(i)] Compute $\eta(T,V)$ for all $T\in\TT_\ell$, if $\eta(T,V)\leq \eps$ for all $T\in\TT_\ell$, stop.
  \item[(ii)] Find $T_0\in\TT_\ell$ with maximal $\eta(T_0,V)$.
  \item[(iii)] Bisect $T_0$ with newest-vertex-bisection to generate $\TT_{\ell+1}$ and goto~(i) (no mesh closure at this point).
 \end{itemize}
Output: In case algorithm terminates at $\ell\in\N$, it produces a mesh $\TT_\eps:=\TT_\ell$ with $\eta(T,V)\leq \eps$ for all $T\in\TT_\eps$.
\end{algorithm}

The following lemma states that Algorithm~\ref{alg:greedy} produces the minimal mesh to satisfy the given  tolerance $\eps$ in the maximum  norm. Since Algorithm~\ref{alg:greedy} does not perform mesh-closure, we define $\T_{\rm nc}\supset \T$ as the set of meshes which can be generated from $\TT_0$ by iterated newest-vertex-bisection without mesh-closure.
\begin{lemma}\label{lem:hopt}
Let $\TT\in\T_{\rm nc}$ denote a mesh with $\max_{T\in\TT}\eta(T,V)\leq \eps$, then $\TT$ is a refinement of $\TT_\eps$ generated by Algorithm~\ref{alg:greedy}. In this case, Algorithm~\ref{alg:greedy} terminates.
\end{lemma}
\begin{proof}
First, assume that Algorithm~\ref{alg:greedy} terminates and produces some mesh $\TT_\eps$.
 Assume that $\TT$ is not a refinement of $\TT_\eps$.
 By the binary tree structure of newest-vertex-bisection, this implies the existence of a descendant $T\in\TT_\eps\setminus\TT$ of some element $T'\in\TT$. In this case, however, $\eta(T',V)$ must have been picked for refinement in Step~(ii) of Algorithm~\ref{alg:greedy} in some intermediate step $k\in\N$ (otherwise it would not have been refined). This, however, implies $\eta(T',V)>\eps$ and hence contradicts the definition of $\TT$.
 
 Second, if Algorithm~\ref{alg:greedy} does not terminate, we obtain a sequence of meshes $\TT_\ell$ which is refined arbitrarily often. This implies that we may define $\TT_\eps:=\TT_\ell$ for sufficiently large $\ell\in\N$ such that $\TT$ is not a refinement of $\TT_\eps$. Then, the arguments of the first part of the proof apply analogously.
\end{proof}

In the following, we emulate Algorithm~\ref{alg:greedy} with a deep RNN. The major obstacle is that we usually can not compute the required error estimator $\eta(T,V)$ exactly only using deep RNNs ($\eta(T,V)$ will contain some $L^2$-norm and hence must be approximated). To circumvent this, we assume the existence of a random variable $\rho(T,V)\geq 0$ (to be constructed later by means of Monte Carlo sampling) with $ \E(\rho(T,V))\leq q^{-1} \eta(T,V)$ and $\P(\rho(T,V)\geq q\eta(T,V))\geq \gamma>0$
for some $q>0$ and $\gamma>0$. The following algorithm is similar to Algorithm~\ref{alg:greedy} with high probability as shown in Lemma~\ref{lem:stochopt} below.
\begin{algorithm}\label{alg:greedystoch}
 Input: Function $V\in H^1(D)$, tolerance $\eps>0$, initial mesh $\TT_0$\\
 For $\ell = 0,1,2,\ldots$ do:
 \begin{itemize}
  \item[(i)] For all $T\in\TT_\ell$ do:
  \begin{itemize}
  \item[(a)] Sample $\rho(T,V)$ $K$-times. 
  \item[(b)] If all samples satisfy $\rho(T,V)\leq \eps$ add $T$ to $\TT_{{\rm stop}}$.
  \item[(c)] Otherwise, add $T$ to $\MM_\ell$.
  \end{itemize}
  \item[(ii)] Remove all elements $T$ from $\MM_\ell$ for which there exists $T'\in\TT_{\rm stop}$ with $T\subseteq T'$.
  \item[(iii)] Generate $\TT_{\ell+1}$ by refining the marked elements $\MM_\ell$ with newest vertex bisection and mesh closure.
 \end{itemize}
Output: The algorithm terminates if $\TT_\ell=\emptyset$ and outputs $\TT_\eps$.
\end{algorithm} 
% 
% Given $T\in\TT_\infty$, let $L_\eps(T,V)\in\N$ denote the minimal refinement level of an element $T'\in\TT_\infty$ with $|T'\cap T|>0$ such that $\eta(T',V)\leq \eps$.

\begin{lemma}\label{lem:stochopt}
 Let $\TT_\eps$ denote the output of Algorithm~\ref{alg:greedystoch} and $\TT_{\delta}'$ denote the output of 
 Algorithm~\ref{alg:greedy} with
 $\delta>0$. Then, there holds
 \begin{align*}
  \P\big(\forall T\in\TT_\eps:\,\eta(T,V)\leq \eps/q\big)\geq 1-(L+1)\#\TT_{\eps/q}'2^{-m},
 \end{align*}
 where $L\in\N$ is the number of iterations of Algorithm~\ref{alg:greedy} needed to produce $\TT_{\eps/q}'$ and $m=K/|\log(1-\gamma)|$. Furthermore, there holds $\P(\#\TT_\eps \leq C \max\{m,\#\TT_{\delta}'\})\geq 1-2^{-C\max\{m,\#\TT_{\delta}'\}}$,
where $C$ depends only on the shape regularity of $\TT_0$ and $\delta\simeq \eps/K$ with hidden constants depending additionally on $q$.
\end{lemma}
\begin{proof}
%For $C>0$, define two error cases of Algorithm~\ref{alg:greedystoch}:
%begin{itemize}
% \item[(E1)] $\eta(T,V)> \eps/q$ but $T$ is added to $\TT_\eps$
% \item[(E2)] $\eta(T,V)\leq \eps/(qC)$ but $T$ is not added to $\TT_\eps$
%\end{itemize}
In the following, meshes with a dash, e.g., $\TT'$, always denote meshes generated by Algorithm~\ref{alg:greedy}.
To compare Algorithm~\ref{alg:greedystoch} with Algorithm~\ref{alg:greedy}, we want to bound the probability of
\begin{align*}
\text{(E1):}\quad \eta(T,V)> \eps/q\quad\text{ but }T\text{ is added to }\TT_{\rm stop}.
\end{align*}
If (E1) does not occur during the runtime of the algorithm, then the output $\TT_\eps$ satisfies
\begin{align*}
\eta(T,V)\leq \eps/q \quad\text{for all }T\in\TT_\eps.
\end{align*}
% Hence, Lemma~\ref{lem:hopt} implies that  $\TT_\eps$ is a refinement of $\TT_{\eps/q}'$ (the output of Algorithm~\ref{alg:greedy}).
The sampling procedure in Step~(i) of Algorithm~\ref{alg:greedystoch} ensures that 
(E1) happens in step $\ell$ with probability less than $\#\TT_\ell (1-\gamma)^{K}$. Thus, to ensure that (E1)
does not occur with high probability over the runtime 
of the algorithm, we choose $K=  |\log(1-\gamma)|m$. 
For $L\in\N$ iterations of Algorithm~\ref{alg:greedystoch}, this guarantees that (E1) occurs with probability less than
\begin{align*}
 \sum_{\ell=0}^L \#\TT_\ell 2^{-m}\leq (L+1)\#\TT_L2^{-m}.
\end{align*}
If we set $L$ to the number of iterations of Algorithm~\ref{alg:greedy} to produce $\TT_{\eps/q}'$, then we conclude $\P\big(\forall T\in\TT_\eps:\,\eta(T,V)\leq \eps/q\big)\geq 1-(L+1)\#\TT_L2^{-m}$ since without (E1), $\TT_\eps$ is a refinement of $\TT_{\eps/q}'$.

To estimate the number of additional refinements compared to Algorithm~\ref{alg:greedy}, we want to bound the probability of 
\begin{align*}
  \text{(E2):}\quad \eta(T,V)\leq q\eps/C\quad\text{ but }T\text{ is not added to }\TT_{\rm stop}.
\end{align*}
Note that (E2) occurs if $\rho(T,V)>\eps$ for at least one of $K$ independent samples despite $\eta(T,V)\leq q\eps/C$.
Markov's inequality shows for all $C>0$
 \begin{align*}
 \P(\rho(T,V)\geq C q^{-1}\eta(T,V))\leq 1/C.
 \end{align*}
 Therefore, the probability that (E2) occurs for an element $T\in\TT_\ell$ is bounded by $1-(1-1/C)^K$. 
  Since $\TT_{q\eps/C}'$ is the coarsest mesh to satisfy $\eta(T,V)\leq q\eps/C$ for all its elements, (E2) can only occur on elements of $\TT_{q\eps/C}'$ or refinements of them.  
  
Assume that (E2) occurs on $r\in\N$ elements before Algorithm~\ref{alg:greedystoch} terminates. 
Then, together with the mesh-closure estimate from~\cite{stevenson08}, the final mesh contains less than $C_{\rm cl}(\#\TT_{q\eps/C}'+r)$ elements, where $C_{\rm cl}>0$ depends only on the shape regularity $\TT_0$.

We will now consider refinement forests $\FF$, which are rooted in $\TT_{q\eps/C}'$. These are $\#\TT_{q\eps/C}'$ binary trees with root nodes corresponding to the elements of $\TT_{q\eps/C}'$.
We denote the total number of leaves of $\FF$ by $N(r)$. The leaves of a forest $\FF$ correspond  to a refinement $\TT$ of $\TT_{q\eps/C}'$. Not every binary forest corresponds to a conforming mesh without hanging nodes, but every newest-vertex-bisection mesh can be represented by at least one ordered binary forest (ordered in the sense that every parent node of every tree has exactly zero children or exactly one \emph{left} and one \emph{right} child). The number of different binary trees with $n$ leaves is given by the Catalan number $C_{n-1}$, where $C_n:=\binom{2n}{n}/(n+1)$, see, e.g.,~\cite[Example 5.3.12]{stanley}. Thus, the total number $M(r,s)$ of different forests $\FF$ with $s$ roots and $N(r)\geq s$ leaves is given by
\begin{align*}
 M(r,s):= \sum_{i_1+\ldots+i_s = N(r)-s} C_{i_1}C_{i_2}\cdots C_{i_s}.
\end{align*}
A combinatorial identity (see, e.g.~\cite{catalan}) shows
\begin{align*}
 M(r,s)&=\begin{cases}
          \frac{s(N(r)-s+1)(N(r)-s+2)\cdots(N(r)-s/2-1)}{2(N(r)-s/2+2)(N(r)-s/2+3)\cdots N(r)} C_{N(r)-s/2} & s\text{ is even}\\
          \frac{s(N(r)-s+1)(N(r)-s+2)\cdots(N(r)-(s+1)/2)}{2(N(r)-(s-3)/2)(N(r)-(s-3)/2+1)\cdots N(r)} C_{N(r)-(s+1)/2} & s\text{ is odd}
         \end{cases}\\
&\leq \frac{s}{2}C_{N(r)-\lceil s/2\rceil}
 \leq\frac{s}{2N(r)-s+1}\binom{2(N(r)-\lceil s/2\rceil)}{N(r)-\lceil s/2\rceil}\leq \frac{s(2e)^{N(r)-\lceil s/2\rceil}}{2N(r)-s+1}\leq (2e)^{N(r)},
\end{align*}
where we used $\binom{n}{k}\leq (e n/k)^k$ in the last estimate.

The probability $\P(\FF)$ that one particular forest $\FF$ occurs can be calculated as follows: Due to Step~(ib) of Algorithm~\ref{alg:greedystoch}, there is exactly one opportunity for (E2) to happen at each node of each tree $\RR$ of $\FF$ (if the element does not get refined by (E2), it never will be refined again by (E2)).

Thus, $\P(\FF)$ is bounded by the probability that $r$ instances of (E2) occur at some nodes of $\FF$. Since a binary tree with at most $N(r)$ leaves has at most $2N(r)-1$ nodes, a binary forest with $s$ roots and $N(r)$ leaves has $2N(r)-s$ nodes. This implies $\P(\FF)\leq \binom{2N(r)-s}{r}(1-(1-1/C)^K)^r$. Thus, the probability $\P(r)$, that Algorithm~\ref{alg:greedystoch} produces an arbitrary forest  with $N(r)$ leaves,  is bounded by
\begin{align*}
 \P(r)&\leq M(r,s)\binom{2N(r)-s}{r}(1-(1-1/C)^K)^r\leq \binom{2N(r)-s}{r}(2e)^{N(r)}(1-(1-1/C)^K)^r.
\end{align*}
Choosing $r\geq\#\TT_{q\eps/C}'$, and $C=\kappa K$, we obtain $(1-(1-1/C)^K)\leq 1/\kappa$ as well as $\binom{2N(r)-s}{r}\leq (4C_{\rm cl} e)^r$. Since $N(r)\leq 2 C_{\rm cl} r$, this shows
$\P(r)\leq (8C_{\rm cl}e^2)^{2C_{\rm cl} r}\kappa^{-r}$.
Finally, the choice $\kappa=(16C_{\rm cl}e^2)^{2 C_{\rm cl}}$ bounds the probability by
\begin{align*}
 \P(r)\leq  2^{-2C_{\rm cl}r}.
\end{align*}
Thus, the probability of $\TT_\eps$ having more than $\max\{C_{\rm cl}(\#\TT_{q\eps/C}'+m),2C_{\rm cl}\#\TT_{q\eps/C}'\}$ elements is thus bounded by
\begin{align*}
 \sum_{r=\max\{m,\#\TT_{q\eps/C}\}}^\infty 2^{-N(r)}\lesssim 2^{-\max\{C_{\rm cl}(\#\TT_{q\eps/C}'+m),2C_{\rm cl}\#\TT_{q\eps/C}'\}}.
\end{align*}
This concludes the proof.
\end{proof}

\begin{lemma}\label{lem:MCerrest}
 For $V\in L^2(\Dom)$ and $T\in\TT$, define
 \begin{align*}
  \rho(T,V):= \Big(\frac{|T|}{N}\sum_{i=1}^N\big(V(\phi_T(x_i)) - \frac{1}{N}\sum_{j=1}^N V(\phi_T(y_j))\big)^2\Big)^{1/2},
 \end{align*}
where $x_1,\ldots,x_N$ and $y_1,\ldots,y_N$ are uniformly i.i.d. points on the reference element $T_{\rm ref}$ and $\phi_T\colon T_{\rm ref}\to T$ is the affine transformation. Then, there holds
\begin{align*}
 \E\rho(T,V)^2 =(1+1/N)\eta(T,V)^2:= (1+1/N)\norm{V-\Pi_T^0 V}{L^2(T)}^2.
\end{align*}
Assume that $V^2\in L^\infty_w(D)$. Then, there exists $C,\gamma>0$ such that each $T\in\TT_\infty$ satisfies
\begin{align*}
  \P(\rho(T,V)^2\geq (1+1/N)\eta(T,V)^2\geq \gamma.%\quad%\text{or}\quad \eta(T,V)^2\leq C
  %|T|h_T^{k}.
\end{align*}
The constants $\gamma$ and $C$ depend only on $N$, $\norm{V^2}{L^\infty_w(D)}$, and the constant $C$ from Lemma~\ref{lem:EX}.
\end{lemma}
\begin{proof}
 Expansion of $\rho(T,V)$ as well as the independence of the Monte Carlo points show
 \begin{align*}
  N\E\rho(T,V)^2 &=\sum_{i=1}^N\E\Big(|T|V(\phi_T(x_i))^2 - \frac{2|T|}{N}V(\phi_T(x_i))\sum_{j=1}^NV(\phi_T(y_j)) + \frac{|T|}{N^2}\sum_{j,k=1}^NV(\phi_T(x_j))V(\phi_T(x_k))\Big) \\
  &=
  \sum_{i=1}^N \Big(\int_T V^2\,dx -2|T|^{-1}(\int_T V\,dx)^2
  +|T|^{-1}(1-\frac{1}{N})(\int_T V\,dx)^2+\frac1N\int_TV^2\,dx\Big)\\
  &=N\Big((1+1/N)\int_T V^2\,dx - |T|(1+1/N)(\Pi_T^0 V)^2\Big)
  =
  N(1+1/N)\norm{V-\Pi_T^0 V}{L^2(T)}^2.
 \end{align*}
To show the second statement, we employ Lemma~\ref{lem:EX}. To that end, note that $X:=\rho(T,V)^2$ is a
 non-negative random variable on $\Omega:= \bigcup_{i=1}^N T_{{\rm ref},i}\subset \R^{d}$,
 where $T_{{\rm ref},i}$ are pairwise disjoint copies of the reference element scaled to $|T_{{\rm ref},i}|=1/N$. We define $w_\Omega$ analogously to $w$ on $\Omega$ with respect to the transformed singularity set $S_{\Omega}:= \bigcup_{i=1}^N\phi_{T,i}^{-1}(S\cap T)$, where $\phi_{T,i}\colon T_{{\rm ref},i}\to T$ are the affine element mappings.
 Note that there holds ${\rm diam}(T)^dw_\Omega(x)\lesssim w\circ \phi_{T,i}(x)$ by definition of the weight $w$.
 By assumption, we have $V^2\in L^\infty_w(D)$ and hence
 \begin{align*}
 \norm{w_\Omega X}{L^\infty(\Omega)} \lesssim |T| \norm{w_\Omega V^2\circ \phi_T}{L^\infty(\Omega)}\lesssim \norm{V^2}{L^\infty_{w}(D)},
 \end{align*}
where we used $|T|\simeq {\rm diam}(T)^d$.
 %Moreover, by assumption
% on $V$, there holds $X\in W^{k,\infty}_{w}(\Omega)$ with respect to the singular set $\widetilde S:=\phi_T^{-1}(S)^{2N}$. This follows from the fact that $|D^k (V\circ \phi_T)|\simeq h_T^{k} |T||D^k V|$ pointwise almost everywhere.
Hence, Lemma~\ref{lem:EX} applies and proves $\P(X\geq q\E(X))\gtrsim 1$. This concludes the proof.
% We consider two cases. First, assume $|X|_{W^{k,\infty}_w(\Omega)}\leq C\E(X)$ for some universal constant $C>0$. Then, we have $\P(X\geq q\E(X))^\alpha\gtrsim 1$ and conclude the proof. Second, if $|X|_{W^{k,\infty}_w(\Omega)} > C\E(X)$, a scaling argument shows
% \begin{align*}
%  \norm{X}{L^1(\Omega)}\lesssim |X|_{W^{k,\infty}_w(\Omega)}\lesssim h^{k} |V|_{W^{k,\infty}(D)}^2
% \end{align*}
% and hence
% \begin{align*}
%  \eta(T,V)^2 
%  \lesssim h_T^{k}\norm{V}{W^{k,\infty}_w(T)}^2.
% \end{align*}
%Altogether, this concludes the proof.
\end{proof}

Finally, we complete Algorithm~\ref{alg:greedystoch} by replacing the theoretical error estimator $\rho(T,V)$ by the concrete DNN approximation $\rho_T$.
\begin{algorithm}\label{alg:complete}
 Input: Function $V\in H^1(D)$, tolerance $\eps>0$, initial mesh $\TT_0$\\
 For $\ell = 0,1,2,\ldots$ do:
 \begin{itemize}
  \item[(i)] For all $T\in\TT_\ell$ do:
  \begin{itemize}
  \item[(a)] Sample $\rho_T$ $K$-times. 
  \item[(b)] If all samples satisfy $\rho_T\leq \eps$ add $T$ to $\TT_{{\rm stop}}$.
  \item[(c)] Otherwise, add $T$ to $\MM_\ell$.
  \end{itemize}
  \item[(ii)] Remove all elements $T$ from $\MM_\ell$ for which there exists $T'\in\TT_{\rm stop}$ with $T\subseteq T'$.
  \item[(iii)] Generate $\TT_{\ell+1}$ by refining the marked elements $\MM_\ell$ with newest vertex bisection and mesh closure.
 \end{itemize}
Output: The algorithm terminates if $\MM_\ell=\emptyset$ and outputs $\TT_\eps:=\TT_\ell$.
\end{algorithm} 

\begin{theorem}\label{thm:generrest}
Let $u\in H^1(D)$.
Given $\eps>0$, let $v_\eps$ denote an approximation to $\nabla u$ which satisfies $\norm{\nabla u - v_\eps}{L^2(D)}\leq \eps$ as well as $v_\eps^2\in L^\infty_w(D)$.
We assume that the refinement indicator $\rho_T$ satisfies $|\rho_T- \rho(T,v_\eps)|\leq \eps/4$ with
$\rho(T,\cdot)$ from Lemma~\ref{lem:MCerrest}. We denote
 by $\TT_\eps$ the output of Algorithm~\ref{alg:complete} as well as by $\TT_{\delta}'$ the output of 
 Algorithm~\ref{alg:greedy} (with $V:=v_\eps$). Then, there holds
 \begin{align*}
  \P\big(\forall T\in\TT_\eps:\,\norm{\nabla u - \Pi_{\TT_\eps}^0 \nabla u}{L^2(T)}\leq (1+2/q)\eps\big)\geq 1-(L+1)\#\TT_{2\eps/q}'2^{-m},
 \end{align*}
 where $L\in\N$ is the number of iterations of Algorithm~\ref{alg:greedy} needed to produce $\TT_{2\eps/q}'$ and $m=K/|\log(1-\gamma)|$. Furthermore, there holds $\P(\#\TT_\eps \leq C \max\{m,\#\TT_{\delta}'\})\geq 1-2^{-C\max\{m,\#\TT_{\delta}'\}}$,
where $C$ depends only on the shape regularity of $\TT_0$, $N$ in the definition of $\rho(T,\cdot)$, and $\norm{v_\eps^2}{L^\infty_w(D)}$. There holds $\delta\simeq \eps/K$ with hidden constants depending additionally on $q$.
\end{theorem}
\begin{proof} 
%We may assume $\eta(T,V)> \eps$  on all $T\in\TT_0$ (otherwise we just remove those elements from the mesh).
We define
\begin{align*}
 \widetilde \rho_T:=\begin{cases}
                     \rho(T,v_\eps) & \rho(T,v_\eps)<\eps/2,\\
                     \rho_T &\text{ else.}
                    \end{cases}
\end{align*}
Note that $\widetilde \rho_T\leq \eps$ if and only if $\rho_T\leq \eps$. Hence, Algorithm~\ref{alg:complete} produces exactly the same output when we replace $\rho_T$ with $\widetilde \rho_T$. We prove 
\begin{align}\label{eq:equiv}
\widetilde \rho_T/2\leq \rho(T,v_\eps)\leq 2\widetilde \rho_T.
\end{align}
To see this, we only need to consider the case $\rho(T,v_\eps)\geq \eps/2$. There holds $\widetilde \rho_T = \rho_T\leq \rho(T,v_\eps)+\eps/4 \leq 2 \rho(T,v_\eps)$ as well as $\rho(T,v_\eps)\leq \rho_T + \eps/4\leq \rho_T + \rho(T,v_\eps)/2$. This concludes~\eqref{eq:equiv}.

 Lemma~\ref{lem:MCerrest} confirms that $\E(\widetilde\rho_T) \leq 2q^{-1}\eta(T,v_\eps)$ as well as $\P(\widetilde\rho_T\geq q/2 \eta(T,v_\eps))\geq\gamma$ are satisfied
for all elements $T\in\TT_\infty$.
 Thus, Lemma~\ref{lem:stochopt} applies and concludes \begin{align*}
  \P\big(\forall T\in\TT_\eps:\,\eta(T,v_\eps)\leq 2/q\eps\big)\geq 1-(L+1)\#\TT_{2\eps/q}'2^{-m},
 \end{align*}
 as well as $\P(\#\TT_\eps \leq C \# \TT_{ \delta}')\geq 1-2^{-C\#\TT_{\delta}'}$. Note that $\TT_\delta'$ is generated by Algorithm~\ref{alg:greedy} with $V=v_\eps$.
 By assumption $\norm{\nabla u -v_\eps}{L^2(D)}\leq \eps$, we have $\norm{\nabla u - \Pi_{\TT_\eps}^0 \nabla u}{L^2(T)}\leq \eta(T,v_\eps) + \eps$ and hence conclude the proof.
% The well-known mesh closure estimate from~\cite{stevenson07} confirms that the number of additionally generated elements
% due to mesh closure is uniformly bounded by the number of marked elements.
% The number of marked elements is bounded by $\#\TT_{\eps/R}'$ since $R^{-1}\rho_T\leq\rho(T,v_\eps)$. 
% Hence, we also obtain $\P(\#\TT_\eps \leq 2\#\TT_{\eps/(CK)}'\geq 1-(e/3)^{\#\TT_{\eps/(CK)}'}$ from~\ref{lem:stochopt} for some universal $C>0$. This 
%and concludes the proof.
\end{proof}

% \begin{algorithm}\label{alg:completernn}
%  Input: Tolerance $\eps>0$, $K\in\N$, initial mesh $\TT_0$\\
%  For $\ell = 0,1,2,\ldots$ do:
%  \begin{itemize}
%   \item[(i)] For all $T\in\TT_\ell$ generate uniform i.i.d. samples $x_{T,1},\ldots,x_{T,KN}\in T$.
%   \item[(ii)] Apply $\texttt{ADAPTIVE}$ to the input sequence $\big((f|_T,x_{T,1},\ldots,x_{T,KN})\big)_{T\in\TT_\ell}$
%   and compute output sequence $\by\in\R^{\#\TT_\ell}$.
%   \item[(iii)] Define $\MM_\ell:=\set{T\in\TT_\ell}{y_T>0}\setminus\TT_{\ell-1}$ (or  $\MM_\ell:=\set{T\in\TT_\ell}{y_T>0}$ for $\ell=0$).
%   \item[(iv)] Generate $\TT_{\ell+1}$ by newest vertex bisection with mesh closure on the marked elements $\MM_\ell$.
%  \end{itemize}
% Output: The algorithm terminates if $\MM_\ell=\emptyset$ and outputs $\TT_\eps:=\TT_\ell$.
% \end{algorithm} 

\begin{proof}[Proof of Theorem~\ref{thm:if}]
Again we denote the meshes generated by Algorithm~\ref{alg:greedy} by $\TT_\delta'$.
We aim to apply Theorem~\ref{thm:generrest} with $N=1$.
 Under the assumption that the input $v_\eps$ satisfies $\norm{\nabla u - v_\eps}{L^2(D)}\leq \eps$,
 we use summation and taking the absolute value to construct a DNN $\rho_T$ which computes
 \begin{align*}
  \rho_T:=|T|^{1/2}\odot\big|v_\eps(x) -  v_\eps(y)\big|.
 \end{align*}
The DNN $\rho_T$ takes as input two samples of uniform i.i.d. points $x_{T},y_{T}\in T$ as well as the square root of the element area $|T|^{1/2}$ (this could also be computed via DNNs but we aim to keep the construction simple).
The approximate multiplication $\odot$ is realized via ${\rm MULTIPLY}$ and achieves accuracy $\mathcal{O}(\eps')$ with a DNN of the size $\mathcal{O}(|\log(\eps')| +
|\log(\norm{v_{\eps'}}{L^\infty(\Dom)})|)$ (see Proposition~\ref{prop:multiply}). Choosing $\eps'\simeq \eps$ sufficiently small, we ensure $|\rho_T-\rho(T,v_\eps)|\leq \eps/4$.

To denote the independent samples required from $\rho_T$ we define $\rho_{T,k}$ as the $k$-th sample computed with
input $x_{T,k},y_{T,k}\in T$.
We construct the RNN {\rm ADAPTIVE} via
\begin{align*}
 y_{T_i}={\rm ADAPTIVE}(\bx_{T_i}):= \max\{\max_{k=1,\ldots,K}\rho_{T_{i},k}, \eps\}-\eps,
\end{align*}
where $\bx_{T_i}:=(x_{T,k},y_{T,k},|T|^{1/2},\eps)_{k=1,\ldots,K}$ and $K\simeq m$.

This already proves the complexity bound in Theorem~\ref{thm:if}.
Furthermore, Algorithm~\ref{alg:RNN2} with ${\rm ADAPTIVE}$ is equivalent to Algorithm~\ref{alg:complete} with $\rho_T$. Hence Theorem~\ref{thm:generrest} applies with $q=1/2$ and proves 
\begin{align*}
  \P\big(\forall T\in\TT_\eps:\,\norm{\nabla u - \Pi_{\TT_\eps}^0 \nabla u}{L^2(T)}\leq 5\eps\big)\geq 1-(L+1)\#\TT_{4\eps}'2^{-m},
 \end{align*}
as well as $\P(\#\TT_\eps \leq C \max\{m,\#\TT_{\delta}'\})\geq 1-2^{-C\max\{m,\#\TT_{\delta}'\}}$ with $\TT_{\delta}'$  denoting the output of Algorithm~\ref{alg:greedy} with $V=v_\eps$ and $\delta\simeq \eps/m$. Let $\TT$ denote a mesh which satisfies~\eqref{eq:opt2} for minimal $N\in\N$ with $N^{-s-1/2}\leq \delta/2$, i.e.
\begin{align*}
 \eta(T,\nabla u)\leq N^{-s-1/2}\leq \delta/2\quad\text{for all }T\in\TT.
\end{align*}
Then, we have $\eta(T,v_\eps)\leq \delta/2+ \eps/(Cm)\leq \delta$ and
Lemma~\ref{lem:hopt} implies that $\TT$ is a refinement of $\TT_{\delta}'$, i.e., $\#\TT_\delta'\leq \#\TT\leq N$.  This shows that with probability larger than $(1-(L+1)\#\TT_{4\eps}'2^{-m})(1-2^{-C\max\{m,\#\TT_{\delta}'\}})$, there holds
\begin{align*}
\# \TT_\eps \max_{T\in\TT_\eps}\norm{\nabla u - \Pi_{\TT_\eps}^0 \nabla u}{L^2(T)}^{1/(s+1/2)}\lesssim \max\{m,N\}(5\eps)^{1/(s+1/2)}.
\end{align*}
Since $N-1\lesssim(\eps/m)^{-1/(s+1/2)}$, we conclude for $m\leq \#\TT_\eps$
\begin{align*}
 \#\TT_\eps \max_{T\in\TT_\eps}\norm{\nabla u - \Pi_{\TT_\eps}^0 \nabla u}{L^2(T)}^{1/(s+1/2)}\lesssim m^{1/(s+1/2)}
\end{align*}
 with probability larger than 
 \begin{align*}
  (1-(L+1)\#\TT_{4\eps}'2^{-m})(1-2^{-C\max\{m,\#\TT_{\delta}'\}})
  \geq 1-((L+1)\#\TT_{4\eps}'+1)2^{-Cm}.
 \end{align*}
Note that $\TT(4\eps)$ as defined in Section~\ref{sec:main2} is conforming and thus a refinement of $\TT_{4\eps}'$ (according to Lemma~\ref{lem:hopt}). Since $L$ is the number of iterations of Algorithm~\ref{alg:greedy}, we know that the maximal level of elements in $\TT_{4\eps}'$ is equal to $L$. This concludes the proof.
\end{proof}

\section{Numerical Experiments}\label{sec:numerics}

\subsection{Hardcoded deep RNN}
As a first experiment, we implement the RNN ${\rm ADAPTIVE}$ from 
Theorem~\ref{thm:refinement} exactly
as shown in the proofs of Section~\ref{sec:construction}. We run 
Algorithm~\ref{alg:RNN} on an L-shaped domain shown in 
Figure~\ref{fig:refinement}. 
We choose a constant right-hand side $f=1$ and start from a coarse triangulation 
with six elements. 
Figure~\ref{fig:convrate} shows that the adaptive method reaches the expected 
convergence rate of $\mathcal{O}(N^{-1/2})$, while the uniform 
approach only achieves a suboptimal rate due to the singularity at the 
re-entrant corner of the domain. Figure~\ref{fig:refinement} compares the 
adaptive meshes generated by ${\rm ADAPTIVE}$ to a standard adaptive mesh 
generated by Algorithm~\ref{alg:adaptive}.
This experiment's main purpose is to show that round-off errors do not spoil the 
theoretically shown performance. 
\begin{figure}
\psfrag{estimator}{\tiny estimator}
\psfrag{number of elements}{\tiny number of elements $\#\TT$}
\psfrag{adaptive}{\tiny adaptive}
\psfrag{uniform}{\tiny uniform}
 \includegraphics[width=0.5\textwidth]{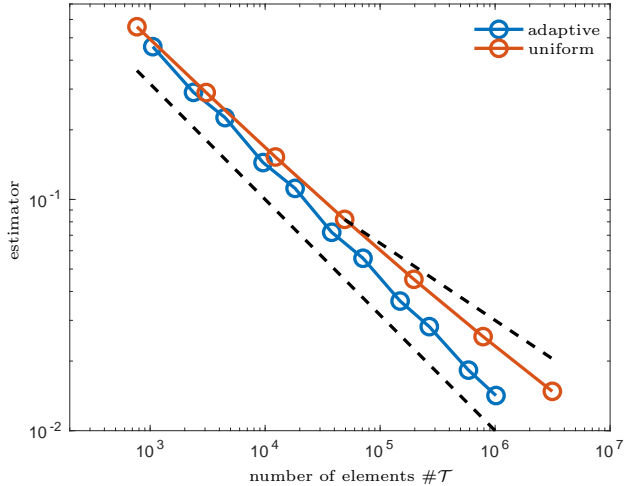}
 \caption{Comparision of the performance of the RNN {\rm ADAPTIVE} versus 
uniform mesh refinement. The dashed lines indicate the expected rates for 
uniform/adaptive refinement $\mathcal{O}(N^{-1/3})$ and 
$\mathcal{O}(N^{-1/2})$.}
 \label{fig:convrate}
\end{figure}
\begin{figure}
 \includegraphics[width=0.49\textwidth]{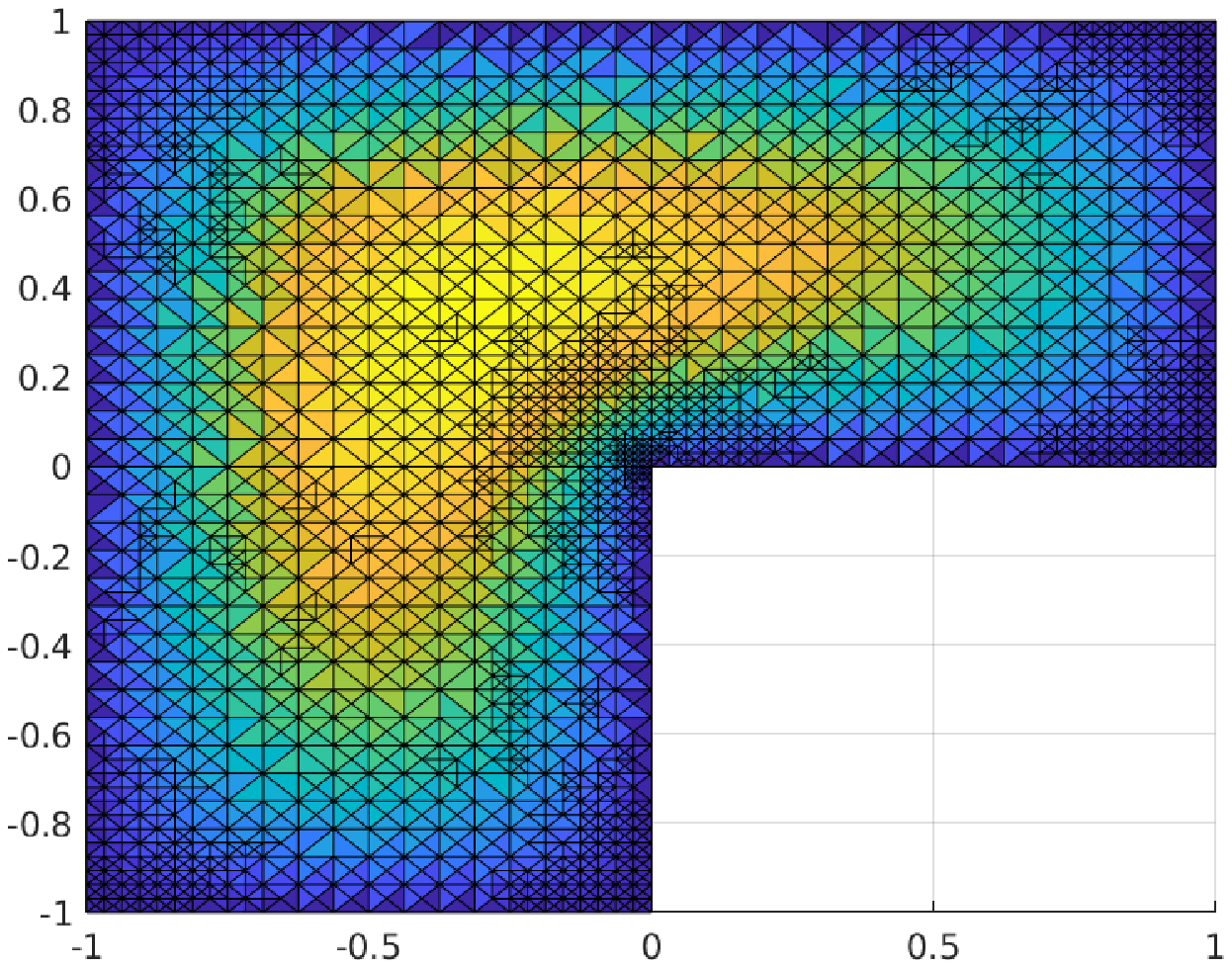}%
  \includegraphics[width=0.49\textwidth]{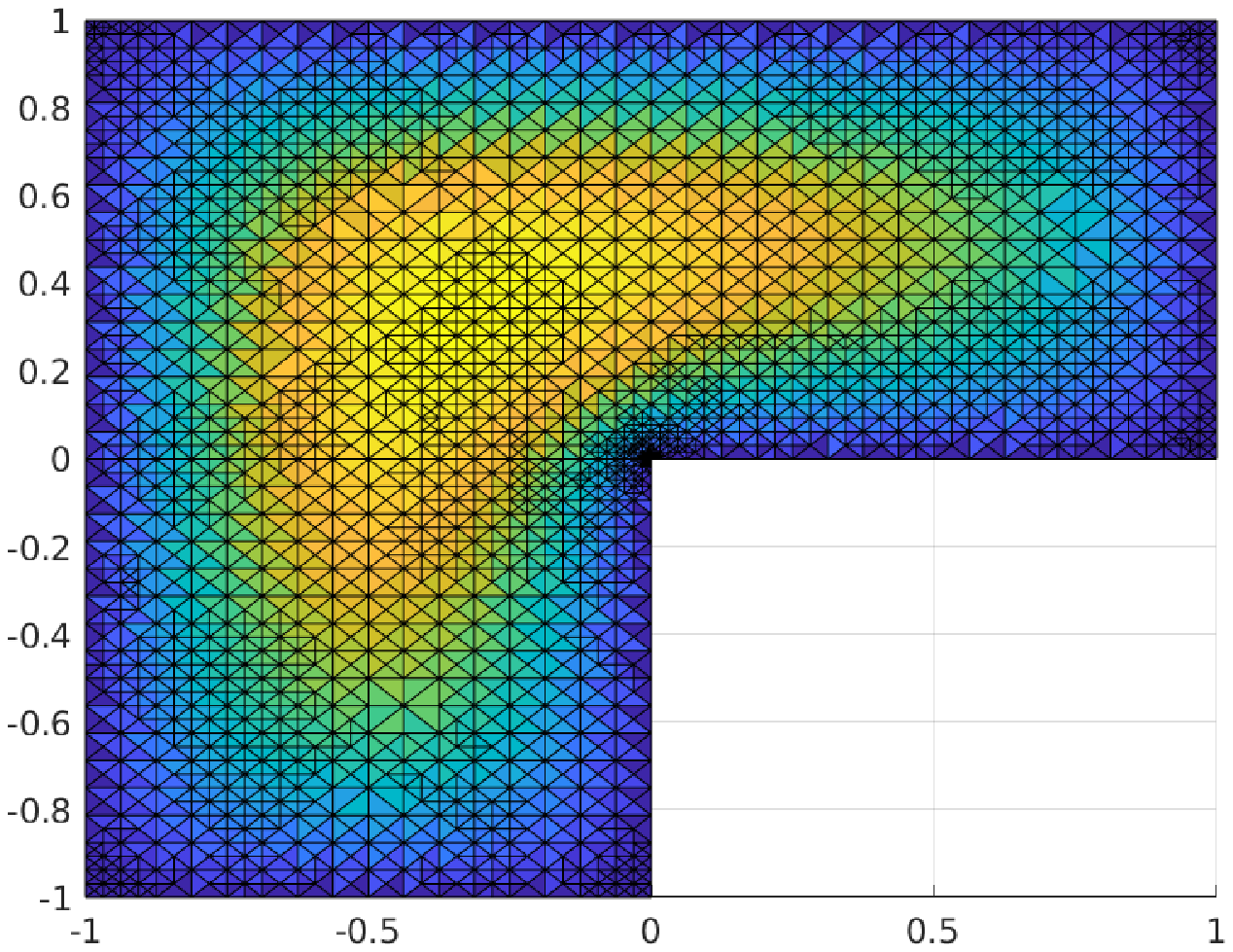}
 \caption{Comparision of the adaptive meshes generated after 9 steps of adaptive 
refinement with the RNN {\rm ADAPTIVE} (left) and the standard residual error 
estimator~\eqref{eq:errest} with D\"orfler marking (right).}
 \label{fig:refinement}
\end{figure}
\bigskip

 The more interesting experiment would be to find the the weights of ${\rm 
ADATIVE}$ by means of computational optimization (machine learning) as described 
in Section~\ref{sec:conclude}.
We do not cover this topic in its entirety, because the training of RNNs is a challenging topics on itself. However, we achieve some intermediate goals in the following two sections.
  
 \subsection{Learning the maximum strategy}

First, we try to find an RNN $\widetilde{\rm MARK}$ which, given 
the exact residual based error estimator from~\eqref{eq:errest}, marks elements 
and achieves the optimal order of convergence. To that end, we use the smallest 
possible blue print for an RNN such that it can represent the maximum strategy (which is much simpler than D\"orfler marking). 
The maximum strategy defines the set of marked elements as
\begin{align*}
 \MM:=\set{T\in\TT}{\rho_T> (1-\theta)\max_{T'\in\TT}\rho_{T'}}.
\end{align*}
Although it is not known wether the maximum strategy leads to optimal 
convergence in the sense of~\eqref{eq:opt}, it is usually observed in practice 
and~\cite{instanceopt} even shows optimality for a slight variation of this 
strategy. 
The maximum strategy can be realized by the combination of two basic RNNs. 
First, $B_1$ is defined for an input $\bx\in\R^{\#\TT}$, $x_i=\rho_{T_i}$ and 
output $\by\in \R^{2\times \#\TT}$ by
\begin{align*}
 y_i=B_1(x_i,y_{i-1})=(x_i,\max(x_i,y_{i-1,2})) = 
 \begin{pmatrix}
  1 & -1 & 0\\
  1 &-1 & 1 \\
 \end{pmatrix}
\phi\left(
\begin{pmatrix}
 1 & 0\\
 -1 & 0\\
 -1 & 1
\end{pmatrix}
\begin{pmatrix}
x_i\\ y_{i-1,2}
\end{pmatrix}
\right).
\end{align*}
Initialization with $y_{0,2}=0$ results in $y_{i,1}=x_i$
as well as $y_{i,2}=\max_{1\leq j\leq i}x_i$ for all $1\leq i\leq \#\TT$.
Then, we initialize a second basic RNN $B_2$ with $\bx\in \R^{2\times\#\TT}$, 
$x_{i,1}=y_{i,1}$ and $x_{i,2}=y_{\#\TT,2}$ for $i=1,\ldots,\#\TT$ (note that if 
we insist on the inizialization as discribed in Section~\ref{sec:rnn}, we need a 
third RNN to copy the value of $y_{\#\TT,2}$ to the entire vector). We filter 
the marked elements by
\begin{align*}
 y_i=B(x_i)=\max(x_{i,1}-(1-\theta)x_{i,2},0)
\end{align*}
and observe that $\MM=\set{T_i\in\TT}{y_i>0}$. Now, we know the structure 
necessary to represent the maximum strategy. 

\bigskip

To find $\widetilde{\rm MARK}$ by machine learning, we set up a simple optimization 
algorithm to compute the 
necessary weights. We initialize an RNN with the structure as given above with 
random weights, run Algorithm~\ref{alg:adaptive} with Step~(3) replaced by our 
RNN as long as $\#\TT_\ell\leq 2\cdot 10^4$, and apply simultaneous perturbation 
stochastic approximation (SPSA) to maximize the energy norm of the finest 
computed solution (note that due to Galerkin orthogonality, maximizing the 
energy norm is equivalent to minimizing the error). %Note that we additionally 
%applied the random shuffling of the input from Theorem~\ref{thm:refinement} both 
%for training and evaluation.

The SPSA approach is basically a stochastic gradient descent algorithm which 
replaces the gradient by a finite difference in a random direction (see, 
e.g.~\cite{spsa} for details). As discussed in the previous section, this is 
necessary since marking is not a continuous procedure. As discussed in 
Section~\ref{sec:conclude}, we limited the values of the recursive weights to 
the set $\{-1,0,1\}$ to avoid blow-up or dampening. The weights found by the 
algorithm for $B_1$ and $B_2$ are
\begin{align*}
 \begin{pmatrix}
  -0.2471& 0.1095 & -0.2358\\
  -0.1868 &0.3123 & -0.9564 \\
 \end{pmatrix}
,\quad
\begin{pmatrix}
 0.4394 & 0\\
 -0.6591 & 0\\
 -0.6466 & -1
\end{pmatrix}
,
\quad
\begin{pmatrix}
-0.1585 & 0.2804.
\end{pmatrix}
\end{align*}
While we cannot offer a meaningful explanation of the marking strategy found, we 
observe in Figure~\ref{fig:RNNmesh} and Figure~\ref{fig:RNNconv} that it behaves 
in an empirically optimal fashion and also the generated meshes look reasonable.

\begin{figure}
 \includegraphics[width=0.7\textwidth]{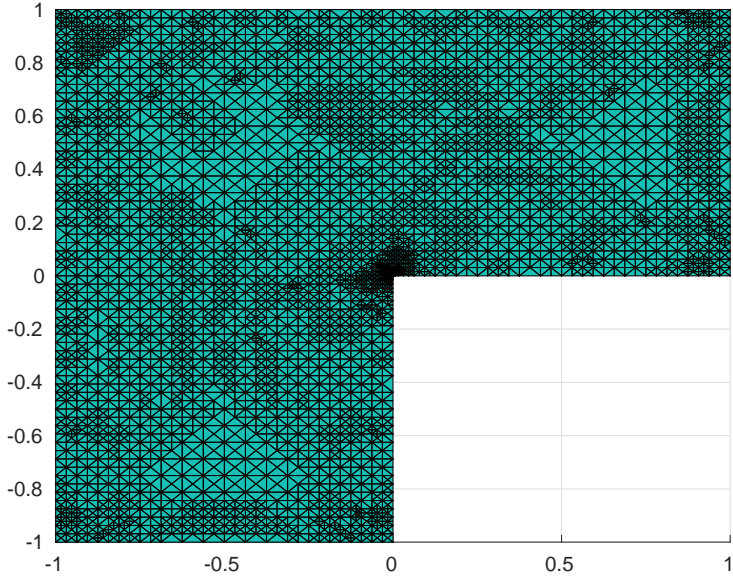}
 \caption{Adaptive mesh generated by the RNN $\widetilde{\rm MARK}$ found by 
stochastic gradient descent.}
 \label{fig:RNNmesh} 
\end{figure}

\bigskip

\begin{figure}
\psfrag{estimator}{estimator}
\psfrag{number of elements}{number of elements $\#\TT$}
 \includegraphics[width=0.7\textwidth]{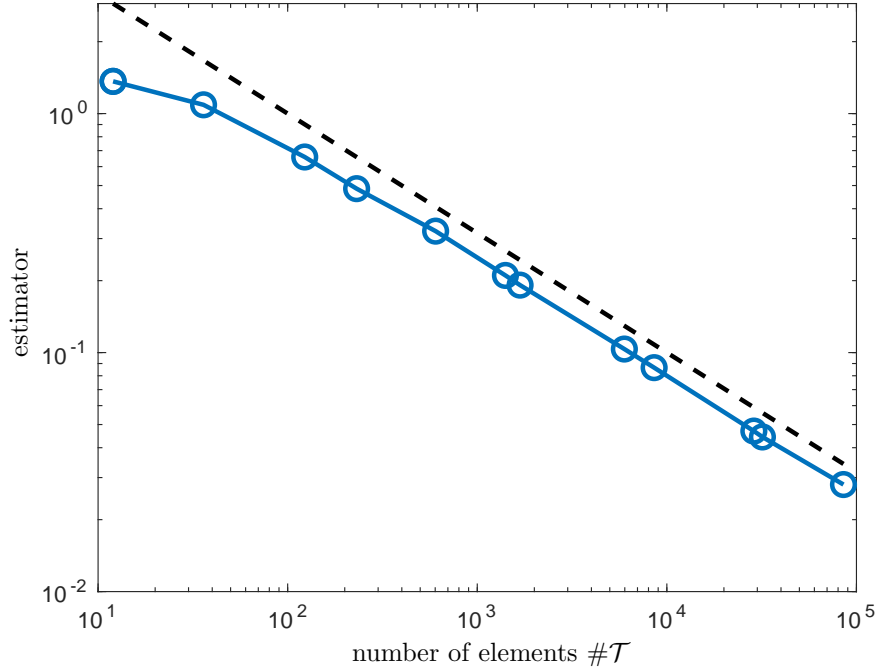}
 \caption{Performance of the RNN $\widetilde{\rm MARK}$ as a marking strategy in 
Algorithm~\ref{alg:adaptive}. The dashed line marks the optimal rate of 
convergence $\mathcal{O}(N^{-1/2})$.}
 \label{fig:RNNconv}  
\end{figure} 
\subsection{Training on the job}\label{sec:totj}
Finally, we try to learn the full deep RNN {\rm ADAPTIVE} for the Poisson problem
\begin{align*}
 -\Delta u &= f\quad\text{in }\Dom,\\
 u&=0\quad\text{on }\partial\Dom.
\end{align*}
We choose a $Z$-shaped domain $\Dom$ to increase the rate difference between optimal and uniform refinement, i.e., $\Dom:=[-1,1]^2\setminus {\rm conv}\{(0,0),(-1,0),(-1,-1/5)\}$.
There are many different plausible methods of how to train the network. We chose to set up a deep RNN which consists of two RNNs $B$ and $B'$ with the layer structure $(s_0,\ldots,s_3)=(16,10,10,10)$ and $(s_0',s_1',s_2') =(11,10,1)$. The two RNNs are only mildly recursive in the sense that $y_i=B(x_i,y_{i-1,1})$ (and analogously for $B'$) only has access to the first component of the previous vector valued output. The input data of $B$ is the sequence $x_1,\ldots,x_{\#\TT}$, where each $x_i$ contains $\nabla U_\TT|_{T}\in\R^2$ for $T=T_i$ as well as for each neighbor element $T$ which shares an edge with $T_i$. Moreover, $x_i$ contains the coordinates of the nodes of $T$ as well as the midpoint evaluation of the right-hand side $f$. This results in $x_i\in\R^{15}$. Since the RNN $B$ also processes the first component of the previous output, the first layer has to consist of $16$ nodes. For the same reason, the first layer of $B'$ has one more node than the final layer of $B$. We call this section \emph{training on the job} because the training of the neural network is part of the adaptive algorithm:
\begin{algorithm}\label{alg:totj}
	\textbf{Input: } Initial mesh $\TT_0$, number of training steps $n_{\rm train}\in\N$.\\
	For $\ell=0,1,2,\ldots$ do:
	\begin{enumerate}
		\item Compute discrete approximation $U_{\ell}$.
		\item For $k=1,\ldots,n_{\rm train}$
		\begin{enumerate}
		\item Apply $\by={\rm ADAPTIVE}(\bx)$.
		\item Use newest-vertex-bisection to refine $\#\TT_\ell/5$ elements
$T_i\in\TT_\ell$ with largest $y_i$ to obtain  $\widetilde \TT_{\ell+1}$.
\item Optimize weights of ${\rm ADAPTIVE}(\bx)$ with the goal to maximize $\norm{\nabla \widetilde U_{\ell+1}}{L^2(\Dom)}$.
\end{enumerate}
\item Apply $\by={\rm ADAPTIVE}(\bx)$.
\item Use newest-vertex-bisection to refine $\#\TT_\ell/5$ elements
$T_i\in\TT_\ell$ with largest $y_i$ to obtain  $\TT_{\ell+1}$.
	\end{enumerate}

\end{algorithm}
The result of this algorithm is shown in Figure~\ref{fig:nn}. 
\begin{remark}\label{rem:learning}
 We note that Algorithm~\ref{alg:totj} is different to Algorithm~\ref{alg:RNN} or~\ref{alg:RNN2} due to the fact that we always refine 20\% of the elements and, in case of Algorithm~\ref{alg:RNN2} we we do not exclude elements which are not refined in a specific step. This is largely to simplify the training process of the network and to minimize the implementational overhead. The optimization step (2c) consists of trying $n_{\rm train}=50$ random perturbations of the network and choosing the best one. To avoid to optimize towards networks which always refine all elements (this would give the largest increase in energy $\norm{\nabla \widetilde U_{\ell+1}}{L^2(\Dom)}$), we restrict ourselves to refining exactly 20 \% of the elements. We are confident that a more sophisticated training method would improve the results. This, however, is beyond the scope of this work.
\end{remark}

\begin{figure}
\psfrag{nnadaptive}{\tiny{\rm ADAPTIVE}}
\psfrag{adaptive}{\tiny uniform}
\psfrag{uniform}{\tiny adaptive}
\psfrag{error}{ error}
\psfrag{nelements}{number of elements $\#\TT$}
 \includegraphics[width=0.7\textwidth]{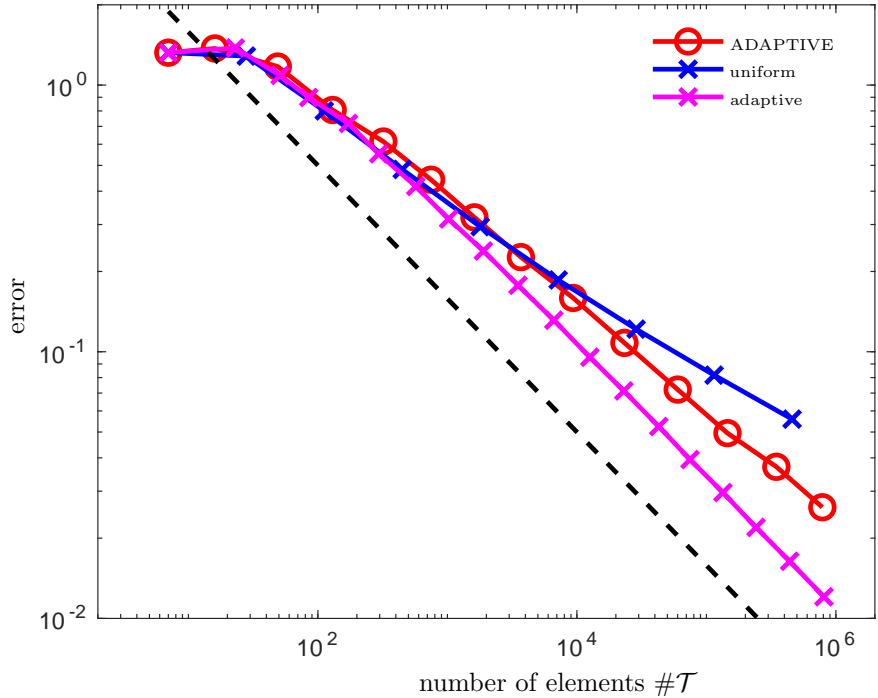}
 \caption{Result (in terms of the residual error estimator plotted over the number of elements) of the \emph{training on the job} Algorithm~\ref{alg:totj} compared with uniform refinement and (optimal) adaptive refinement via the residual error estimator. We observe that the deep RNN {\rm ADAPTIVE} (which is trained on the fly) clearly beats the uniform refinement and almost achieves optimal rate of convergence. Even with the very crude learning method described in Remark~\ref{rem:learning}, the deep RNN approach leads to an advantage over uniform refinement without using any information about the problem such as error estimators. This improved rate implies particularly that as long as the training cost is proportional to the cost of the solve step, this approach will eventually also be more cost effective than plain uniform refinement.}
 \label{fig:nn}
\end{figure}

\bibliographystyle{plain}
\bibliography{literature}
\end{document}